\documentclass[12pt]{amsart}

\usepackage{amssymb, amsthm, hyperref}
\usepackage{enumerate, amsfonts, amsxtra,amsmath,latexsym,epsfig, color}
\usepackage{mathrsfs}

\usepackage{hyperref}

\usepackage{xypic}

\input xy
\xyoption{all}

\newtheorem {theorem}{Theorem} [section]
\newtheorem {lemma} [theorem] {Lemma}
\newtheorem {proposition} [theorem] {Proposition}
\newtheorem {corollary} [theorem] {Corollary}

\newtheorem {question}[theorem] {Question}

\newtheorem {observation} [theorem] {Observation}

\theoremstyle{definition}
\newtheorem{definition}[theorem]{Definition}
\newtheorem{remark}[theorem]{Remark}
\newtheorem{example}[theorem]{Example}

\def\co{\colon\thinspace}

\def\into{\, \hookrightarrow\, }
\def\onto{\, \twoheadrightarrow\, }

\def\R {\mathbb R}

\def\mc {\mathcal}
\def\F {\mathbb F}
\def\Z {\mathbb{Z}}
\def\res {\mathscr{R}}
\def\catC {\mathscr{C}}
\def\catCSA {\catC_{\mathrm{csa}}}
\def\Obj {\mathrm{Obj}}
\def\Mor {\mathrm{Mor}}

\def\Hom {\mathrm{Hom}}
\def\Aut {\mathrm{Aut}}
\def\Out {\mathrm{Out}}

\def\SK {\underrightarrow{\text{Ker}}\, }

\def\Mod {\mathrm{Mod}}

\def\nc{\mathrm{nc}}
\def\ab{\mathrm{ab}}

\def\x{\underline{x}}
\def\y{\underline{y}}

\def\model {\mathscr{M}}
\def\rmodel {\model} 

\def\wh {\widehat}

\begin{document}

\title[The structure of limit groups]
{The structure of limit groups over hyperbolic groups}

\author[Daniel Groves]{Daniel Groves}
\address{Department of Mathematics, Statistics, and Computer Science\\University of Illinois at Chicago\\322 Science and Engineering Offices (M/C 249)\\851 S. Morgan St.\\
Chicago, IL 60607-7045\\USA }
\thanks{The work of the first author was supported by the National Science Foundation and by a grant from the Simons Foundation (\#342049 to Daniel Groves)}
\email{groves@math.uic.edu}

\author[Henry Wilton]{Henry Wilton}
\address{DPMMS\\Centre for Mathematical Sciences\\Wilberforce Road\\Cambridge\\CB3 0WB\\UK}
\thanks{The second author was supported by the EPSRC}
\email{h.wilton@maths.cam.ac.uk}

\date{\today}

\maketitle

\begin{abstract}
Let $\Gamma$ be a torsion-free hyperbolic group.  We study $\Gamma$--limit groups which, unlike the fundamental case in which $\Gamma$ is free, may not be finitely presentable or geometrically tractable.  We define \emph{model} $\Gamma$--limit groups, which always have good geometric properties (in particular, they are always relatively hyperbolic).  Given a strict resolution of an arbitrary $\Gamma$--limit group $L$, we canonically construct a strict resolution of a model $\Gamma$--limit group, which encodes all homomorphisms $L\to \Gamma$ that factor through the given resolution.  We propose this as the correct framework in which to study $\Gamma$--limit groups algorithmically.  We enumerate all $\Gamma$--limit groups in this framework.
\end{abstract}

Limit groups over a group $\Gamma$ (otherwise known as $\Gamma$--limit groups) arise naturally when studying algebraic geometry over $\Gamma$; that is to say, sets of homomorphisms $\Hom(G,\Gamma)$, where $G$ is a finitely generated group.   We will be concerned with the case in which $\Gamma$ is a non-elementary, torsion-free hyperbolic group, a case which has been studied extensively by Sela \cite{sela:diohyp} and others \cite{kharlampovich_decidability_2013,Perin,perinetal14,GrovesWilton10}. The results of \cite{sela:diohyp} extend Sela's solution to Tarski's problems in the case when $\Gamma$ is free in \cite{sela:dio1} \emph{et seq.}\ (see also \cite{kharlampovich_irreducible_1998} \emph{et seq.}).

Since every finitely generated subgroup of $\Gamma$ is a $\Gamma$--limit group, the hyperbolic case immediately presents new problems that do not arise in the free case.  One such problem is that hyperbolic groups typically contain subgroups that are finitely generated but not finitely presented \cite{rips_subgroups_1982}.  Thus, it is not immediately clear how to give a finite description of a $\Gamma$--limit group.    Moreover, $\Gamma$--limit groups do not in general have the nice geometric properties of limit groups over free groups (which are all toral relatively hyperbolic \cite{Alibegovic05,Dahmani03}, in particular finitely presented).  The geometry of hyperbolic and relatively hyperbolic groups has been integral to the solutions of many algorithmic problems \cite{RipsSela:CanonicalReps,sela:isomorphism,dahmani09,dahmanigroves1,dahmanigroves2}) and has been very useful in the algorithmic study of limit groups (see, for example, \cite{GrovesWilton09,GrovesWilton10}). 

Our goal in this paper is to associate to a $\Gamma$--limit group $L$ a \emph{model} $\Gamma$--limit group $M$, which is toral relatively hyperbolic, contains $L$ as a subgroup, and most importantly, encodes a large set of homomorphisms $L\to \Gamma$ in a natural way.  (The reader should bear in mind that the interest of $\Gamma$--limit groups arises precisely because they encode large sets of homomorphisms to $\Gamma$; indeed, Sela proved that $\Gamma$--limit groups are precisely the finitely generated fully residually $\Gamma$ groups.)  To give the reader a flavour of our results, we state some properties of a model $M$ associated to a $\Gamma$--limit group $L$.

\begin{theorem}\label{thm: Precis}
Let $L$ be a freely indecomposable $\Gamma$--limit group.  There exists a $\Gamma$--limit group $M$ such that:
\begin{enumerate}
\item\label{Precis:LtoM injective} there is an injection $\eta\co L\into M$;
\item\label{Precis:model TRH} $M$ is toral relatively hyperbolic, in particular finitely presented;
\item\label{Precis:Modular} there is a natural homomorphism of modular groups \[\Phi:\Mod(L)\to\Mod(M);\]
\item\label{Precis:Intertwine} $\Phi$ intertwines $\eta$ (that is, 
\[
\eta(\alpha(g))=\Phi(\alpha)(\eta(g))
\]
for $\alpha\in\Mod(L)$ and $g\in L$).
\end{enumerate}
\end{theorem}

Recall that the  \emph{modular automorphisms} of a  $\Gamma$--limit group are those that arise naturally from the JSJ decomposition of $\Gamma$.  They play a special role in the theory of $\Gamma$--limit groups, and so items (3) and (4) above demonstrate that $M$ captures much of the information contained in $L$.
Items \eqref{Precis:LtoM injective} and \eqref{Precis:model TRH} are proved in Theorem \ref{t:Model props} below.  Items \eqref{Precis:Modular} and \eqref{Precis:Intertwine} are proved in Theorem \ref{t:Properties of models}.

We now elaborate on the details of our construction. According to \cite[Section 2]{sela:diohyp}, all of the homomorphisms from $L$ to $\Gamma$ can be described using a collection of \emph{strict} resolutions (see Definition \ref{d: Strict resolution} below).  These take the form of a sequence of \emph{strict} maps
\[
L_0\to L_1\to\cdots\to L_n
\]
where $L_0$ is (a quotient of) $L$, the $L_i$ are all $\Gamma$--limit groups and $L_n$  admits a (fixed) strict map to $\Gamma$.  In particular, Sela shows that every $\Gamma$--limit group admits a strict resolution, and we use these to build model $\Gamma$--limit groups canonically.

It is important to note that we need to relax two of Sela's conditions.  Sela insists that the maps in resolutions are surjective, but we allow our resolutions to include non-surjective maps.  Moreover, Sela terminates with a $\Gamma$--limit group which is a free product of a free group and freely indecomposable subgroups of $\Gamma$.  We do not know how to ensure this and keep the features of our construction.  These changes seems to be necessary both in order to make model $\Gamma$--limit groups geometrically tractable and also to have their structure closely mimic that of the $L_i$. 

\begin{theorem} \label{thm: Precis 2}
To every strict resolution $L_0\to L_1\to\ldots\to L_n$ there is canonically associated a strict resolution that fits into a commutative diagram\\
\centerline{
\xymatrix{
   L_0  \ar@{>}[d]^{\eta_0} \ar@{>}[r] & L_1 \ar@{>}[d]^{\eta_1} \ar@{>}[r] & \cdots \ar@{>}[r] & L_n \ar@{>}[d]^{\eta_n}  \\
 M_0 \ar@{>}[r] & M_1\ar@{>}[r] &  \cdots\ar@{>}[r] & M_n}}\\\\
where each $(M_i, \eta_i)$ satisfies the properties of Theorem \ref{thm: Precis}.
\end{theorem}
This result follows from Theorems \ref{t:Model props} and \ref{t:Properties of models}, along with our construction of models from a resolution of $L_0$.

Suppose that $\res = L_0 \stackrel{\lambda_0}{\to} L_1 \stackrel{\lambda_1}{\to} \ldots \stackrel{\lambda_{n-1}}{\to} L_n$ is a strict resolution of $L_0$.
The way in which  $\res$ describes a large set of homomorphisms from $L_0$ to $\Gamma$ is via homomorphisms which {\em factor through $\res$}.  This means that a homomorphism can be described as a composition of modular automorphisms of the $L_i$ and the given maps $\lambda_i$, finally composed with a particular embedding of each of the freely indecomposable free factors of $L_n$ into $\Gamma$ and an arbitrary map from the free free factor of $L_n$ to $\Gamma$.  It is important to note that one needs modular automorphisms of $L_n$ in this description.  From a homomorphism factoring through $\res$, the intertwining maps $\Phi_i \co \Mod(L_i) \to \Mod(M_i)$ canonically induce a homomorphism factoring through the model resolution.

We give a brief outline of this paper.  In Section \ref{s:KM} we discuss some differences between the work in this paper and a construction that appears in \cite{kharlampovich_decidability_2013} and also in \cite{KMT}.  In Section \ref{s:Prelim} we recall some background required for the results in this paper.  
In Section \ref{s:Defn models} we give the definition of the (sequence of) model(s) built from a strict resolution of a $\Gamma$--limit group $L$, in terms of the universal property that a model satisfies.  In Section \ref{s:Gluing} we provide the construction of models inductively in terms of a graph of groups built from the Grushko and JSJ decompositions of the limit groups appearing in the resolution of $L$, and prove this construction satisfies the universal property from Section \ref{s:Defn models}.
In Section \ref{s:Model Properties} we prove the remaining properties of model $\Gamma$--limit groups.  In Section \ref{s:Rel Imm} we introduce the notion of {\em relatively immutable} subgroups of toral relatively hyperbolic groups, prove basic results about them and provide an effective enumeration of them.  In Section \ref{s:Calc QCE} we prove that quasi-convex enclosures (see Subsection \ref{ss:SA}) can be algorithmically computed.  Finally, in Section \ref{s:Enumeration} (See Theorem \ref{thm: Enumerate Gamma-limit groups}) we apply the construction of models to algorithmically enumerate $\Gamma$--limit groups via an enumeration of all of the pairs of resolutions as in Theorem \ref{thm: Precis 2}.
This enumeration is analogous to the enumeration of limit groups (over free groups) which we provided in \cite{GrovesWilton09}.

{ \subsection*{Acknowledgements}{We thank the referee for a very careful reading and for many very helpful comments and corrections which improved the manuscript in numerous ways.}}

\section{A cautionary example} \label{s:KM}

A somewhat similar construction to the one we perform in this paper is described in \cite[$\S\S3.5$]{kharlampovich_decidability_2013} and \cite[$\S\S6.6$]{KMT}.  (The construction in these two papers appears to be identical; we refer to the numbering in \cite{kharlampovich_decidability_2013}.)

In order to highlight some differences between the construction in \cite{kharlampovich_decidability_2013, KMT} and the one in this paper, in this section we describe certain examples.  It is worth remarking that these examples are very elementary, and the issues that they exemplify are entirely generic -- they will occur for many torsion-free hyperbolic groups $\Gamma$ and many  $\Gamma$--limit groups.

We start with a basic fact about hyperbolic groups.  For brevity, we call a torsion-free hyperbolic group \emph{rigid} if it admits no non-trivial free or cyclic splitting.

\begin{lemma}\label{lem:Rigid embedding}
Any torsion-free word-hyperbolic group $G$ embeds in a rigid word-hyperbolic group $\Gamma$.  Furthermore, if $G$ is virtually special (in the sense of Haglund--Wise \cite{Haglund08}) then { we can choose $\Gamma$ to be} virtually special.
\end{lemma}
\begin{proof}
We may assume that $G$ is non-elementary.  Choose pairwise non-conjugate elements $h_1,\ldots,h_n\in G$  (with $n\geq 2$) such that, in every nontrivial free or virtually cyclic splitting of $G$,  some $h_i$ is {hyperbolic}.  (Finitely many $h_i$ suffice by standard accessibility results.)  By the ping-pong lemma, replacing the elements $h_i$ with proper powers, we may assume that the subgroup $H$ that they generate is free and quasiconvex.   Now fix any non-elementary hyperbolic group $G_0$ without free or cyclic splittings, and let $H_0$ be a free, quasiconvex, malnormal subgroup of $G_0$ of the same rank as $H$.  (Examples of such pairs abound.)  The amalgam
\[
\Gamma:=G*_{H\cong H_0} G_0
\]
is now the fundamental group of a $2$--acylindrical graph of hyperbolic groups with quasiconvex edge groups, and hence is hyperbolic by \cite{BF:combination}.  In any free or cyclic splitting of $\Gamma$, $G_0$ is elliptic by construction, hence $G$ is too. Since $n\geq 2$, no edge is stabilized by $H$, and so the splitting is trivial.

Finally, if $G$ is virtually special, we may also choose $G_0$ to be virtually special, and then $\Gamma$ is virtually special by the main theorem of \cite{Wise}.
\end{proof}

We now describe our examples.

\begin{example} \label{ex:KM}
Suppose that $F$ is a finitely generated nonabelian free group, and that $w$ is an element in $F$ which is not a proper power, and so that $F$ admits no splittings over $\{ 1 \}$ or $\Z$ in which $w$ is elliptic.  Such elements $w$ are entirely generic and easy to find---they are called \emph{rigid} by Cashen--Macura \cite{CashenMacura10}, and the procedure they describe can be used to certify that a given element is rigid.

Let $D$ be the double of $F$ over $w$, and note that $D = F \ast_{\langle w \rangle} F$ is the JSJ of $H$. The group $D$ is an ordinary limit group, since the natural retraction $D\to F$ is strict.  Applying Lemma \ref{lem:Rigid embedding}, we obtain a torsion-free hyperbolic group $\Gamma$ that admits no splittings over $\{1 \}$ or $\Z$ and in which $D$ embeds.

Now let $u$ be a primitive element of $D$ so that $D$ admits no nontrivial splitting over $\{ 1 \}$ or $\Z$ in which $u$ is elliptic, and let 
$D_1$ be the double of $D$ over $\langle u \rangle$.  Let $\rho_1 \co D_1 \to D$ be the map that identifies the two copies of $D$ in $D_1$.  
This is a strict map.

Let $G_1$ be a finite-index subgroup of $D_1$ which does not contain $u$, and let $C=\langle u^k\rangle = G_1 \cap \langle u \rangle$ (so $k > 1$).  Let $G_2 = \rho_1(G_1)$ and note that it is straightforward to check that $G_1 \onto G_2$ is a strict map, so since $G_2$ is a subgroup of $\Gamma$,
\[	G_1 \to G_2 	\]
is a strict resolution of $G_1$.  In the description from \cite{kharlampovich_decidability_2013}, the group $G_2$ is given a fixed (conjugacy class of) embedding into $\Gamma$, which can be chosen to be the inclusion described above.
\end{example}

In this example there are many homomorphisms from $G_2$ to $\Gamma$ that are not injections.  However, by pre-composing the inclusion of $G_2$ into $\Gamma$ with a sequence of Dehn twists in the edge group of the JSJ of $G_2$, we get an infinite sequence of non-conjugate embeddings of $G_2$ into $\Gamma$.  Needless to say, since these maps are all injections, they do not factor through a proper shortening quotient of $G_2$.  Thus, in order to describe all of the homomorphisms from $G_2$ to $\Gamma$, and therefore all of the homomorphisms from $G_1$ to $\Gamma$ that factor through $G_1 \to G_2$, one cannot forget about the JSJ decomposition of $G_2$.

\begin{remark}
It is also straightforward to come up with an example of a pair $G_1 \to G_2$ with a strict map, $G_2$ a subgroup of a torsion-free hyperbolic group with nontrivial JSJ such that $G_2$ has no proper shortening quotients.  One can do this with $3$--manifolds, Property (T) hyperbolic groups, the Rips construction, or other ways.
\end{remark}

We now describe the construction from \cite[$\S\S3.5$]{kharlampovich_decidability_2013} for Example \ref{ex:KM}.  In their notation,  $n =2$ and $G_2 = H_1$.  Note that they {\em explicitly} insist that all of the maps between the $G_i$ are proper quotients, and that the final $H_i$ are given with fixed embeddings of the $H_i$ into $\Gamma$ (up to conjugacy).

The primary JSJ of $G_1$ contains the edge group $\langle u^k \rangle$, whose image in $G_2$ is $u^k$.  {(See \S\S \ref{ss: Primary JSJ} for the definition of the canonical primary JSJ decomposition.)}  The construction from \cite{kharlampovich_decidability_2013} extends the centralizer of $u^k$ in $G_2$.  This extension of centralizers obviously has presentation
\[	\langle G_2 , t \mid [t,u^k] = 1 \rangle \]	

This procedure is performed for each of the edges in the JSJ of $G_1$ (which are $G_1$--conjugacy classes of intersections of $\langle u \rangle$ with $G_1$).  Since there are no abelian or QH subgroups in the JSJ decomposition of $G_1$, and since $n=2$, we finish with
the group $\overline{G}_1$, which is this iterated extension of centralizers.  Clearly $\overline{G}_1$ is a $\Gamma$--limit group and contains $G_1$ as a subgroup.

It is then stated in \cite{kharlampovich_decidability_2013}: 

\ 

``We now extend each subgroup $H_i$ by its quasi-convex closure $\Gamma_i$. Denote the obtained group by $N$.  Therefore $N$ is NTQ and total {\em(sic)} relatively hyperbolic.  Each $\Gamma_i$ is freely indecomposable."

\ 

There is no proof that $N$ is NTQ or toral relatively hyperbolic in \cite{kharlampovich_decidability_2013} or \cite{KMT}, only the above assertion.  We now explain that when applying this construction to Example \ref{ex:KM} neither of these assertions is true.  We do not provide the definition of an `NTQ group' here, but only note that it is a particular kind of $\Gamma$--limit group.  The group $N$ constructed in \cite{kharlampovich_decidability_2013} for Example \ref{ex:KM} is not a $\Gamma$--limit group, let alone an NTQ--group.

For Example \ref{ex:KM}, we have $H_1 = G_2$, and only one {$H_i$}.
It is not clear precisely what ``extend each subgroup $H_i$ by its quasi-convex closure" is supposed to mean, but it seems the only reasonable interpretation in this case is to form the graph of groups
\[	N = \overline{G}_1 \ast_{G_2} \Gamma_1	.	\]
In this example, $\Gamma$ admits no nontrivial splittings over $\{ 1 \}$ or $\Z$ so the description of how to find $\Gamma_i$ in \cite[$\S3.4$]{kharlampovich_decidability_2013} yields $\Gamma_1 = \Gamma$

However, the element $u^k$ has a nontrivial root in $\Gamma$, so if we glue $\Gamma$ onto $\overline{G}_1$ along $G_2$ then the resulting group $N$ is not CSA. Both $\Gamma$--limit groups and toral relatively hyperbolic groups are CSA, so this means that it is not a $\Gamma$--limit group (so certainly not NTQ) and also not toral relatively hyperbolic.  Thus the claim from \cite{kharlampovich_decidability_2013} is false.

Moreover, as described above, the group $N$ cannot be used to describe all of the homomorphisms from $G_2$ to $\Gamma$, nor those from $G_1$ to $\Gamma$ which factor through $G_1 \to G_2$, since not all injective homomorphisms from $G_2$ to $\Gamma$ can be extended to homomorphisms from $\Gamma$ to $\Gamma$.

As we stated above, this is a very simple example, and it exhibits behavior that is entirely generic to the situation for $\Gamma$--limit groups for torsion-free hyperbolic $\Gamma$.

The construction in \cite[$\S\S3.5$]{kharlampovich_decidability_2013} is very important to the construction in that paper, and the above Example shows that Proposition 11 from \cite{kharlampovich_decidability_2013} (which is \cite[Theorem 2]{KMT}) is not proved in either of these papers, because the construction they make does not have the properties that they claim it does.  It appears to be a fundamental construction upon which many of the other algorithms rely (see for example, \cite[$\S\S4.4$]{kharlampovich_decidability_2013}, where they use it to `construct algorithmically a finite number of $\Gamma$--NTQ systems corresponding to branches $b$ of the canonical $\Hom$--diagram\ldots').  In summary, this appears to be a very serious error in \cite{kharlampovich_decidability_2013}.

The construction in this paper takes these issues (and others) into account, and embeds a $\Gamma$--limit group (equipped with a strict resolution) into a toral relatively hyperbolic $\Gamma$--limit group, as described in Theorems \ref{thm: Precis} and \ref{thm: Precis 2}.

\begin{remark}
Since the first version of this paper was circulated, new versions of \cite{kharlampovich_decidability_2013} and \cite{KMT} have been posted to the arXiv, which claim to address the issues outlined here. In the interests of pointing out some of the subtleties inherent in the construction made in this paper, we decided to leave this section largely unchanged.
\end{remark}

\section{Preliminaries} \label{s:Prelim}

Let $Q$ be a group.  A {\em $Q$--limit group} is a limit of finitely generated subgroups of $Q$ in the space of $k$--generated marked groups, for any fixed $k$.  It is convenient to work with the following notions, which help to elucidate the connection between $Q$--limit groups and homomorphisms from a finitely generated group $G$ to $Q$.

\begin{definition}
Suppose that $Q$ and $G$ are groups.  A sequence of homomorphisms $\{ \rho_n \co G \to Q \}$ is {\em convergent}\footnote{Other authors call these `stable' sequences.  See, for example, \cite[Definition 1.6]{BFnotes}.} if for any $g \in G$ either (i) $g \in \ker(\rho_n)$ for all but finitely many $n$; or (ii) $g \not\in \ker(\rho_n)$ for all but finitely many $n$.

The {\em stable kernel} of $\{ \rho_n \}$ is the set $\SK(\rho_n)$ of elements $g \in G$ which are in the kernel of $\rho_n$ for all but finitely many $n$.

A {\em $Q$--limit group} is a group of the form $L = G / \SK(\rho_n)$ where $G$ is a finitely generated group and $\{ \rho_n \co G \to Q \}$ is a convergent sequence of homomorphisms.
\end{definition}

\begin{definition}
A group $G$ is {\em equationally noetherian} if for any finitely generated group $H$ there is a finitely presented group $\wh{H}$ along with an epimorphism $\rho \co \wh{H} \to H$ so that the map
\[	\rho^* \co \Hom(H,G) \to \Hom(\wh{H},G)	\]
induced by precomposition with $\rho$ is a bijection.  

In other words, every homomorphism from $\wh{H}$ to $G$ factors through $\rho$.  Yet another way of expressing this is that once a tuple of elements of $G$ satisfies the finitely many relations for $\wh{H}$ it automatically also satisfies the infinitely many relations for $H$.  Thus, the `Hilbert Basis Theorem' holds for equations over $G$.
\end{definition}

\begin{theorem} [Sela; \cite{sela:diohyp}, Theorem 1.22]
Torsion-free hyperbolic groups are equationally noetherian.
\end{theorem}

\begin{definition}
Suppose that $G$ and $D$ are groups.  We say that $D$ is {\em fully residually--$G$} if for every finite set $A \subset D$ there is a homomorphism $\phi_A \co D \to G$ which is injective on $A$.
\end{definition}

It is well known (see, for example, \cite[Theorem 2.1]{Ould-Houcine}) that if $G$ is equationally noetherian then the class of $G$--limit groups is exactly the class of finitely generated fully residually--$G$ groups.  From this, the following is straightforward.

\begin{corollary}
Let $\Gamma$ be a torsion-free hyperbolic group and let $L$ be a $\Gamma$--limit group.  Then there exists a convergent sequence $\{ \rho_n \co L \to \Gamma \}$ with $\SK(\rho_n) = \{ 1 \}$.
\end{corollary}

\subsection{Abelian subgroups}

For the study of $\Gamma$--limit groups, the following definition and result are particularly useful.

\begin{definition}
A subgroup $H$ of a group $G$ is {\em malnormal} if for all $g \in G \smallsetminus H$ we have $H \cap H^g = \{ 1 \}$.

A group $G$ is called a {\em CSA group} if any maximal abelian subgroup of $G$ is malnormal.

A group $G$ is called {\em commutative transitive} if whenever $g_1, g_2, g_3 \in G$ with $g_2 \ne 1$ and
$[g_1,g_2] = [g_2,g_3] = 1$ then $[g_1, g_3] = 1$.
\end{definition}

Our interest in torsion-free CSA groups comes from the following result, which is due to Sela and implicit in \cite[Section 1]{sela:diohyp}.

\begin{theorem}\label{thm:Sela}
Suppose that $\Gamma$ is a torsion-free hyperbolic group and that $L$ is a $\Gamma$--limit group.
Then $L$ is torsion-free, CSA and all abelian subgroups of $L$ are finitely generated.
\end{theorem}
It is easy to see that a CSA group is commutative transitive. 

\begin{remark}
Being CSA is a closed condition in the space of marked groups, as is being torsion-free.  So these properties are straightforward to prove.  However, the fact that abelian subgroups are finitely generated is much deeper, and is due to Sela.  The reader is referred to \cite[Corollary 5.12]{groves:limitRH2} for a proof.
\end{remark}

\subsection{Primary JSJ decompositions}\label{ss: Primary JSJ}

JSJ decompositions are one of the fundamental tools needed to study $\Gamma$--limit groups.  They were developed in this context by Sela \cite{sela:diohyp}; various other groups of authors have developed related theories \cite{RipsSela97,DunwoodySageev99,FuPap06,Bowditch:cutpoints}, and the whole theory was put into a unified context by Guirardel--Levitt \cite{GuiLevJSJI, GuiLevJSJII,GLCyl2011,GuiLevJSJ:Monster}. 

The specific decomposition that we will work with is called the \emph{canonical primary} JSJ decomposition.   In this section, we shall { adapt} the definitions and results of Guirardel--Levitt to define the canonical primary JSJ decomposition and to explain some of its properties.

\begin{definition} \cite[Definition 3.31]{dahmanigroves1}
Suppose that $G$ is a torsion-free CSA group.  An abelian splitting of $\Gamma$ is called {\em essential} if whenever $E$ is an edge group and $\gamma \in G$ satisfies $\gamma^k \in E$ for some $k \ne 0$ we already have $\gamma \in E$.  An essential splitting is called {\em primary} if every noncyclic abelian subgroup is elliptic.
\end{definition}

In Guirardel--Levitt's terminology, if we let $\mc{E}$ denote the set of root-closed abelian subgroups and $\mc{A}_{\nc}$ the class of non-cyclic abelian subgroups then the Bass--Serre tree of a primary splitting is \emph{an $\mc{E}$--tree relative to $\mc{A}_{\nc}$}.  {Our goal is to construct  a canonical, primary JSJ tree for $\Gamma$-limit groups.  A related JSJ tree for toral relatively hyperbolic groups was used in \cite[\S5]{guirardel_mccool_2015}, in which the Bass--Serre trees of primary splittings are called \emph{RC trees}.  Our JSJ decomposition will differ slightly from the RC-JSJ decomposition of \cite{guirardel_mccool_2015}.

We give the definition and state and prove some of the properties of the primary JSJ decomposition here.

\begin{definition} \cite[Introduction]{GuiLevJSJ:Monster}
A primary tree is \emph{universally elliptic} if every edge stabilizer is elliptic in every primary splitting. A \emph{primary JSJ tree} is a universally elliptic primary tree which is \emph{maximal} among all universally elliptic primary trees, meaning that it dominates any other universally elliptic primary tree.
\end{definition}

When we speak of a \emph{canonical} primary JSJ splitting, we mean one that is invariant under the natural action of the outer automorphism group. The construction of the canonical primary JSJ tree goes via the canonical \emph{abelian} JSJ tree.  An abelian JSJ tree is defined in the same way as a primary JSJ tree, using the set $\mc{A}$ of all abelian subgroups in the place of $\mc{E}$; again, we only consider $\mc{A}$-trees relative to $\mc{A}_\nc$. Guirardel--Levitt proved that a CSA group $G$ (such as  a $\Gamma$-limit group) has a canonical abelian JSJ tree, and described its structure \cite{GLCyl2011}.

\begin{theorem}[Guirardel--Levitt]\label{thm: Abelian JSJ}
Let $G$ be a finitely generated, one-ended, torsion-free, CSA group.  There exists a canonical abelian JSJ tree $T_\ab$, which is bipartite with the following structure:
\begin{enumerate}
\item One class of vertices has maximal abelian stabilizers.
\item The stabilizers of the other class of vertices are either \emph{rigid} or \emph{quadratically hanging (QH)} (defined below).
\item If $e$ is an edge of $T_\ab$ and $x$ is an adjacent, non-abelian vertex, then $G_e$ is maximal abelian in $G_x$.
\item Furthermore, the incident edge stabilizers at $x$ form a malnormal family. 
\end{enumerate}
In particular, $T_\ab$ is 2-acylindrical.
\end{theorem}
\begin{proof}
Existence of the canonical abelian JSJ tree is stated as \cite[Theorem 5]{GLCyl2011}; see also Proposition 6.3 of the same paper, where it is proved that $T_\ab$ (denoted there by $T_c$) is 2-acylindrical.  The bipartite structure follows from its description as a tree of cylinders.  
\end{proof}

The \emph{rigid} vertices have the property that they admit no non-trivial abelian splitting relative to $\mc{A}_\nc$ and their incident edge groups.  Note that below we will consider vertex groups which are rigid with respect to primary splittings, rather than with respect to all abelian splittings, and this will be the meaning of the word rigid after Definition \ref{d:rigid}. 

The \emph{quadratically hanging} vertices are fundamental groups of surfaces, with edge groups attached isomorphically to the infinite cyclic subgroups corresponding to boundary components.

In order to describe the primary JSJ decomposition $\Delta$, we must first introduce the notion of a \emph{socket}.

}
\begin{definition}
Let $\Sigma$ be compact hyperbolic surface with $k\geq 0$ boundary components and let $d_1, \ldots , d_k$ be generators for the subgroups corresponding to the boundary components.  Let $n_1,\ldots,n_k$ be positive integers. The group $S$ obtained from $\pi_1\Sigma$ by attaching an $n_i$-th root to $d_i$ for each $i$ is called a \emph{socket}.  When considered as a subgroup, a socket will always refer to a vertex group in a graph-of-groups decomposition.  If $S$ is a vertex group in a graph of groups, we insist that the incident edge groups are infinite cyclic with generators identified with the roots attached to boundary components, and that every root has an incident edge group attached to it. A socket is called \emph{maximal} if it is not properly contained in any larger socket which appears as a vertex in a graph-of-groups decomposition.
\end{definition}

{
It is a usual feature of JSJ decompositions that the flexible vertices are QH subgroups.  In the case of essential splittings, however, the flexible vertices are sockets, because of the requirement that the edges be root-closed.

There is an important difference between socket groups for abelian splittings and socket groups for primary splittings, which is the difference between our primary JSJ splitting and the RC-JSJ from \cite{guirardel_mccool_2015}.  Edge stabilizers in the RC-JSJ tree are required to be elliptic in every abelian splitting rel $\mc{A}_{\nc}$, whether or not this splitting is RC (primary, in our terminology).  For the primary JSJ tree, we require edge stabilizers to be elliptic in every primary splitting.  The difference can be seen if a vertex group is a once-punctured Klein bottle, with the adjacent edge group corresponding to the puncture.  Cutting along the non-separating, two-sided simple closed curve yields a primary splitting whose edge group is elliptic in all other primary splittings of the once-punctured Klein bottle group (relative to the puncture group), but not in all abelian splittings, as witnessed by cutting along a separating curve. In the RC-JSJ of \cite{guirardel_mccool_2015}, the once-punctured Klein bottle group would be a flexible vertex, whereas in our primary JSJ decomposition, we cut along this curve and replace this once-punctured Klein bottle with two vertex groups -- one corresponding to an annular neighbourhood of the curve, and one corresponding to the rest of the surface, and we have two edge groups corresponding to the two (isotopic) curves that join these two vertex groups.  The new vertex group corresponding to the surface will be a rigid socket group.

As noted in \cite[$\S9.5$]{GuiLevJSJ:Monster}, this distinction happens on orbifolds exactly when the surface is either a Klein bottle, a once-punctured Klein bottle, or a Klein bottle with one orbifold point.  Since we are in the torsion-free CSA setting, the once-punctured Klein bottle is the only case that can arise.  Moreover, if $\Gamma$ is a one-ended torsion-free CSA group and there is an abelian edge group $E$ in some primary splitting of $\Gamma$ which is elliptic in all primary splittings of $\Gamma$ but hyperbolic in some abelian splitting $\Lambda$, then $E$ must correspond to the non-separating, two-sided curve on some once-punctured Klein bottle QH subgroup.  Indeed, we must get a hyperbolic-hyperbolic pair from $E$ and some edge group in $\Lambda$, which leads to a QH subgroup.  If we take a maximal QH subgroup, then there can be no non-separating, two-sided curve intersecting the curve for $E$, which leads us to conclude that the QH subgroup corresponds to a once-punctured Klein bottle.

We refer to these rigid vertex groups obtained from a Klein bottle in the above way as an {\em exceptional} socket.  Note that an exceptional socket is not a maximal socket.
}

{
\begin{definition} \label{d:rigid}
Suppose that $K$ is a group and $\mc{H}$ is a collection of subgroups.  We say that $(K,\mc{H})$ is {\em rigid} if $K$ does not admit any nontrivial primary splitting relative to $\mc{H}$.  If $K$ is a nonabelian vertex group in a graph-of-groups decomposition, we say that $K$ is {\em rigid} if $(K,\mc{K}_e \cup \mc{A}_{nc})$ is rigid where $\mc{K}_e$ is the set of edge groups adjacent to $K$. 

A vertex group of a primary decomposition that is  neither abelian nor rigid is called {\em flexible}.
\end{definition}
}

\begin{proposition}
Suppose that $G$ is freely indecomposable, finitely generated, torsion-free and CSA.  Any flexible vertex of a primary JSJ decomposition of $G$ is a maximal socket.
\end{proposition}
\begin{proof}
This can be proved by following the proof from \cite[$\S6$]{GuiLevJSJ:Monster}, which builds a QH subgroup, then noting that in order to find a primary splitting we need to adjoin the roots to the QH subgroup in order to build a socket.  This is also similar to \cite{sela:SRII,dahmaniguirardel:iso,dahmanigroves1}, and can be proved by similar techniques to any of those papers.
\end{proof}

Note that the converse to the above proposition does not hold: there can also be rigid vertices that are sockets, necessarily built by attaching roots to the boundary components of a thrice--punctured sphere. When discussing the JSJ below, it will still be important to insist that any such vertices are {either exceptional sockets or maximal} sockets.

{
\begin{remark}
Note that our terminology differs from that of \cite{guirardel_mccool_2015}, in which the non-abelian vertices are called \emph{rigid with sockets} and \emph{QH with sockets}.\end{remark}
}
We apply these ideas in our setting.  Let $\Gamma$ be a torsion-free hyperbolic group and let $L$ be a $\Gamma$--limit group.  In particular, $L$ is a torsion-free CSA group and all abelian subgroups of $L$ are finitely generated.

\begin{proposition} \label{p:JSJ props}
If $L$ is a freely indecomposable $\Gamma$--limit group then $L$ has a canonical primary JSJ decomposition $\Delta$.  Furthermore, $\Delta$ has the following properties.
\begin{enumerate}
\item Vertices of $\Delta$ fall into three classes (denoted by $u_i,v_j$ and $w_k$ respectively) with the following properties:
\begin{enumerate}
\item $\Delta_{u_i}$ is \emph{rigid} for each $i$ (and, if it is a socket, is {either an exceptional socket or} a maximal socket);
\item $\Delta_{v_j}$ is a maximal socket, for each $j$;
\item $\Delta_{w_k}$ is abelian.
\end{enumerate}
\item $\Delta$ is bipartite, with only abelian vertices adjacent to rigid or socket vertices.
\item If $v$ is a vertex associated to a nonabelian vertex group $\Delta_v$ then every edge group adjacent to $\Delta_v$ is maximal abelian in $\Delta_v$. 
\item $\Delta$ is $2$--acylindrical.
\end{enumerate}
Conversely, any primary splitting $\Delta$ that satisfies the above properties is the canonical primary JSJ.
\end{proposition}
{
\begin{proof}
We begin by showing that the primary JSJ decomposition exists.  We start with the canonical abelian JSJ tree $T_\ab$, guaranteed by Theorem \ref{thm: Abelian JSJ}, and explain how to construct $T$, the Bass--Serre tree of the canonical primary JSJ decomposition $\Delta$, following \cite[\S5]{guirardel_mccool_2015}.

For each edge $e$ incident at some abelian vertex $w$, let $\overline{G}_e$ be the subgroup consisting of the roots of elements of the edge stabilizer $G_e$.  The union of the translates $\overline{G}_e.e$ is a subtree with every edge adjacent to $w$.  For each edge $e$, fold the corresponding subtree to a single edge $\bar{e}$ incident at $w$, with stabilizer $\overline{G}_e$.  The tree $T'$ is then the minimal $G$-invariant subtree of the result of performing these folds.  (More precisely: $T'$ is trivial if $G$ acts with a global fixed point on the folded tree, and otherwise it is \emph{the} minimal $G$-invariant subtree.)

The other modification to $T_\ab$ which is required is as follows.  Suppose that $\Delta_{v_i}$ is a QH-subgroup of $T_{\ab}$ with underlying surface a once-punctured Klein bottle.  (See the above discussion for an explanation of why the once-punctured Klein bottle is the only relevant case.)  Then $\Delta_{v_i}$ admits a single primary splitting relative to the adjacent edge group, corresponding to cutting the Klein bottle along the two-sided simple closed curve.  Such a splitting is universally elliptic for primary splittings, but not for abelian splittings.  Thus, cut along this curve and replace the once-punctured Klein bottle with an exceptional socket group, along with an extra abelian vertex group corresponding to the curve.

Having performed these modifications we have a tree $T$, which we claim satisfies properties (1)--(4).

We first prove that the resulting tree satisfies properties (1)--(4) from the statement of the result, and then that any splitting which satisfies these properties is a canonical primary JSJ decomposition.

The vertex groups of $T$ come in different kinds, depending on how they arose from the tree $T_{ab}$.  It is clear from the construction that the vertex groups of $T$ can be partitioned into the required types from the statement of (1), and that the resulting structure is bipartite according to the scheme described in (2).  Statement (3) follows immediately from the construction and the corresponding property for $T_{ab}$, and statement (4) follows immediately from statement (3).  
It follows from conditions (2)--(4) that $T$ is equal to its own (collapsed) tree of cylinders, and then it follows from \cite[Corollary 7.4]{GuiLevJSJ:Monster} that $T$ is canonical.

We next prove the converse.    Let $\Delta$ be a primary splitting satisfying items (1) to (4).  We will first show that $\Delta$ is a primary JSJ splitting; to do this, we check first that it is universally elliptic among all primary decompositions, and then that it is maximal among all universally elliptic primary decompositions.

To check that $\Delta$ is universally elliptic, let $\Upsilon$ be a (without loss of generality one-edge) primary splitting and let $e$ be an edge of $\Delta$.  If $\Delta_e$ is non-cyclic then it is elliptic in $\Upsilon$ by hypothesis, so we need only consider the case in which $\Delta_e$ is cyclic. 
Suppose that $\Delta_e$ is not elliptic in $\Upsilon$.  Let $\Lambda$ be a primary JSJ decomposition of $L$.  By \cite[Lemma 2.8]{GuiLevJSJ:Monster} some refinement $\Lambda_0$ of $\Lambda$ dominates $\Delta$.  Since edge stabilizers in $\Delta$ and $\Lambda_0$ are both abelian and $\Lambda_0$ is primary, there is an edge group in $\Lambda_0$ which is conjugate to $\Delta_e$.  However, $\Delta_e$ is not universally elliptic, so this edge group in $\Lambda_0$ must correspond to a simple closed curve $c$ in a surface associated to a socket subgroup of $\Lambda$.  

Let $\Delta_v$ be the nonabelian vertex group adjacent to $\Delta_e$.  The above argument shows that every cyclic edge group of $\Delta$ is elliptic in $\Lambda_0$, and thus every edge group is elliptic, because $\Lambda_0$ is primary.  If $\Delta_v$ is rigid, this implies that $\Delta_v$ is elliptic in $\Lambda$, and hence
must be contained in part of the socket containing the group associated to $c$.  Thus, in this case, $\Delta_v$ is a socket which is neither exceptional nor maximal, which contradicts property (1a).  On the other hand, if $\Delta_v$ is a socket, then it is clearly not maximal, which contradicts property (1b). 
  This proves that $\Delta_e$ is elliptic in $\Upsilon$, and therefore $\Delta$ is a universally elliptic splitting.

To prove that $\Delta$ is maximal universally elliptic, we must prove that all vertex groups of $\Delta$ are elliptic in any universally elliptic primary splitting $\Theta$ (see \cite[Definition, p.4]{GuiLevJSJ:Monster}).  Since the edge groups are elliptic in $\Theta$, it is clear that the rigid vertex groups of $\Delta$ must be elliptic in $\Theta$.  That abelian vertex groups of $\Delta$ are elliptic in $\Theta$ follows quickly from the fact that $G$ is one-ended and CSA, and the fact that $\Theta$ is a primary splitting.

Finally, suppose that $\Delta_{v_j}$ is a socket vertex group of $\Delta$.  Then there is a natural subgroup $H \le \Delta_{v_j}$ (the fundamental group of the surface $\Sigma$ from the definition, before the roots are added).  It is clear that $H$ is a {\em QH--subgroup}, in the sense of \cite[Definition 5.13]{GuiLevJSJ:Monster}.  Therefore, by \cite[Theorem 5.27]{GuiLevJSJ:Monster}, $H$ is elliptic in the JSJ deformation space corresponding to {\em all} abelian edge groups relative to noncyclic abelian edge groups.  However, since we are only considering primary splittings, this implies that the socket of $H$ is elliptic in $\Theta$, as required.  Thus $\Delta$ is maximal universally elliptic, and so $\Delta$ is a  primary JSJ decomposition.

Finally, we note that conditions (2) to (4) imply that the Bass--Serre tree of $\Delta$ is its own (collapsed) tree of cylinders, and hence is a canonical primary JSJ tree by the argument applied to $T$ above.  This completes the proof.
\end{proof}
}

For brevity, we will call the decomposition $\Delta$ produced by the above theorem the \emph{JSJ} decomposition of $L$.  When a freely indecomposable $\Gamma$--limit group $L$ is not a subgroup of $\Gamma$ (and often when it is) it admits a nontrivial primary splitting.  Therefore, unless $L$ is a (socket of a) surface group, the primary JSJ decomposition of $L$ is nontrivial.

Let $\mathcal{H}$ be a family of subgroups of $L$, closed under conjugacy.  We will also sometimes need to work with the Grushko and JSJ decompositions of $L$ \emph{relative to $\mathcal{H}$}.  These decompositions exist for the same reasons as above -- for example, instead of taking the JSJ decomposition for $(\mc{E},\mc{A}_{\nc})$--trees, take the JSJ decomposition for $(\mc{E},\mc{A}_{\nc}\cup \mc{H})$--trees.
The work in \cite{GuiLevJSJ:Monster} works in the relative setting also.

\subsection{Modular automorphisms and envelopes}

One important role of the JSJ is to encode an associated collection of automorphisms of a limit group $L$, namely the \emph{modular group} $\Mod(L)$ (cf.\ \cite[Definition 5.2]{sela:dio1}).  We briefly recall the definition of a Dehn twist.  Suppose a group $G=A*_CB$ and $z$ is in the centralizer $Z(C)$.  Then the assignments  $\delta_z(a)=a$ for all $a\in A$ and $\delta_z(b)=zbz^{-1}$ for all $b\in B$ define an automorphism $\delta_z$ of $G$, called a \emph{Dehn twist} in the splitting $G=A*_CB$. Similarly, if $G=A*_C$ with stable letter $t$, and $z$ is in the centralizer $Z(C)$, then the assignments $\delta_z(a)=a$ for all $a\in A$ and $\delta_z(t)=tz$ define an automorphism $\delta_z$ of $G$, called a \emph{Dehn twist} in the splitting $G=A*_C$.

\begin{definition}\label{d:Gen Dehn twist and modular group}
Let $L$ be a $\Gamma$-limit group, and let $\Delta$ be its canonical primary JSJ decomposition. An automorphism $\alpha$ of $L$ is called a \emph{generalized Dehn twist} if it is of one of the following forms.
\begin{enumerate}
\item A Dehn twist in a one-edge splitting of $L$ obtained by contracting all but one edges of $\Delta$.
\item A Dehn twist in a one-edge splitting of $L$ obtained by cutting along a two-sided simple closed curve in a socket vertex of $\Delta$.
\item Let $\Delta_w$ be an abelian vertex group of $\Delta$. Then $\alpha$ restricted to $\Delta_w$ fixes every edge group incident at $\Delta_w$ (equivalently, $\alpha$ fixes the peripheral subgroup $\overline{P}(\Delta_w)$ defined in Definition \ref{d:Envelope}), and $\alpha$ acts as the identity on every other vertex group and stable letter of $\Delta$.  
\end{enumerate}
The \emph{modular group} of $L$, denoted by $\Mod(L)$, is the subgroup of $\Out(L)$ generated by the equivalence classes of generalized Dehn twists.  The elements of the modular group are called \emph{modular automorphisms}.
\end{definition}

Note that modular automorphisms act as inner automorphisms on rigid vertex groups.  Furthermore, certain slightly larger subgroups associated to a rigid vertex group are also preserved (up to conjugacy) by modular automorphisms.  These larger subgroups were named the \emph{envelopes} of rigid vertices by Bestvina--Feighn \cite{BFnotes}.

Working with the canonical JSJ, we are able to simplify slightly the definition of an envelope.

\begin{definition} \label{d:Envelope}
Let $\Delta$ be the JSJ of $L$, and $A$ an abelian vertex stabilizer.  The subgroup $P(A)\subseteq A$ is generated by the incident edge stabilizers in $A$, and the \emph{peripheral subgroup} $\overline{P}(A)$ is the minimal direct factor of $A$ containing $P(A)$.

The \emph{envelope} of $V$, $E(V)$, is then generated by $V$ together with the peripheral subgroups of the adjacent (abelian) vertex stabilizers. 
\end{definition}

\subsection{Strict  homomorphisms and resolutions} \label{s:strict}

Together with the JSJ decomposition, the second pillar of the structure theory of $\Gamma$--limit groups is provided by the notion of a strict homomorphism.  The following definition is rather lengthy, but the idea is that there exists a homomorphism $L\to R$, where $R$ is a `simpler' limit group, and the homomorphism has prescribed behaviour on each part of the JSJ of $L$.

\begin{definition}[Strict Homomorphism; Sela \cite{sela:dio1}, Definition 5.9]\label{d:strict}
Let $G$ be a finitely generated group, and let $G = G_1 \ast G_2 \ast \cdots \ast G_n \ast \F$ be the Grushko decomposition of $G$, where $\F$ is a free group and the $G_i$ are freely indecomposable.  Suppose that $\Lambda_i$ is a primary splitting of $G_i$.

Let $H$ be a finitely generated group.  A homomorphism $\eta \co G \to H$ is {\em strict with respect to the splittings $\Lambda_i$} if there is a free product decomposition $H = H_1 \ast \cdots \ast H_n$ of $H$ (no assumptions are made about the decomposability of the $H_i$) so that $\eta(G_i) = H_i$ and for each $i$ we have that:
\begin{enumerate}
\item $\eta$ is injective when restricted to envelopes of rigid vertex groups of $\Lambda_i$;
\item The image under $\eta$ of each socket vertex group of $\Lambda_i$ is nonabelian in $H_i$;
\item If $E$ is an edge group of $\Lambda_i$ corresponding to a boundary component of a socket-type vertex then $\eta |_{E}$ is injective; and
\item If $A$ is an abelian vertex group of $\Lambda_i$, and $\overline{P}(A)$ is the peripheral subgroup of $A$ then $\eta |_{\overline{P}(A)}$ is injective.
\end{enumerate}

If $G$ is a  freely indecomposable $\Gamma$--limit group and $\Lambda$ is the primary JSJ decomposition of $G$ then a homomorphism $\eta \co G \to H$ is {\em strict} if it is strict with respect to $\Lambda$.  Note that in the case the $\Lambda_i$ are the canonical primary JSJ splittings, item 4 implies item 3.
\end{definition}

\begin{definition} \label{d:Converge}
Let $G$ be a finitely generated group, $L$ a $\Gamma$--limit group and $\{ f_n \co G \to \Gamma \}$ a
convergent sequence of homomorphisms.  We say that {\em $\{ f_n \}$ converges to $L$} if
\[	L \cong G / \SK \{ f_n \} .	\]
\end{definition}

\begin{definition}\label{d: Strict resolution}
Suppose that $G$ is a finitely generated group, equipped with its Grushko decomposition and primary splittings $\Lambda_i$ of the free factors as in Definition \ref{d:strict}.  A {\em strict resolution of $G$} is a sequence of epimorphisms
\[	 G \onto L_0 \onto L_1 \onto \ldots \onto L_n	,	\]
where $L_n$ is a free product of subgroups of $\Gamma$ and each map $G \onto L_0$ and $L_i \onto L_{i+1}$ is a strict map (with respect to some free and primary splittings of $L_i$).
\end{definition}

The next theorem now characterizes $\Gamma$--limit groups in terms of strict resolutions. Although this characterization is significantly more complicated than the original definition, it comes equipped with much more information.   In particular, it will allow us to describe a $\Gamma$--limit group $G$ inductively, in terms of `simpler' $\Gamma$--limit groups.

\begin{theorem} \cite[Theorem 1.31]{sela:diohyp}, cf. \cite[Theorem 5.12]{sela:dio1} \label{t:strict res implies limit group}
A finitely generated group $G$ is a $\Gamma$--limit group if and only if it admits a strict resolution.

Furthermore, each of the groups $G$ and $L_i$ in the resolution are $\Gamma$--limit groups, and we can take the abelian splittings of the free factors of the $L_i$ to be the canonical primary JSJ decomposition of the free factors. 
\end{theorem}

\begin{remark} \label{rem:different resolution defn}
By Theorem \ref{t:strict res implies limit group}, we lose nothing by assuming up front that the groups $L_i$ in a strict resolution of a finitely generated group $G$ are $\Gamma$--limit groups.

Once we assume that the groups involved are $\Gamma$--limit groups, there is no reason to assume that each of the homomorphisms in a resolution are surjective (since finitely generated subgroups of $\Gamma$--limit groups are themselves $\Gamma$--limit groups).  

Moreover, the terminal group $L_n$ may be assumed merely to admit a strict map to $\Gamma$, rather than be a free product of subgroups of $\Gamma$, since Theorem \ref{t:strict res implies limit group} still implies that $L_n$ is a $\Gamma$--limit group. In this case, we may have two resolutions like $\lambda \co L_n \to \Gamma$ and $L_n \to \lambda_n(L_n) \to \Gamma$.  In case the freely indecomposable free factors of $\lambda(L_n)$ admit nontrivial JSJ decompositions, this second resolution encodes more homomorphisms from $L_n$ to $\Gamma$ than the first one.

These two changes give us some extra flexibility which is central to the constructions in this paper.
\end{remark} 

In this paper, we use a strict resolution $\res$ of a freely indecomposable $\Gamma$--limit group $L$ in order to embed $L$ in a $\Gamma$--limit group $\model_{\res}$ so that $\model_{\res}$ is finitely presented, and furthermore toral relatively hyperbolic.  This allows us to give an algorithmic enumeration of $\Gamma$--limit groups, even though not all $\Gamma$--limit groups need be finitely presented.

We make the following definition in order to capture one of the features of strict homomorphisms in a more general setting.

\begin{definition} \label{d:degenerate}
Suppose that $G$ is a group and $\mathcal P$ is a finite collection of subgroups of $G$.  We say that a homomorphism $\eta \co G \to H$ (for some group $H$) is {\em non-degenerate with respect to $\mc{P}$} if $P \cap \ker(\eta) = \{ 1 \}$ for all $P \in \mathcal{P}$.  Otherwise, we say that $\eta$ is {\em degenerate with respect  to $\mc{P}$}.
\end{definition}

\begin{observation}
Suppose that $L$ is a freely indecomposable $\Gamma$--limit group and that $\mc{D}$ is the set of edge groups of the primary JSJ decomposition of $L$.  If $\eta \co L \to R$ is a strict map then $\eta$ is non-degenerate with respect to $\mc{D}$.
\end{observation}

\subsection{Relative hyperbolicity}

One of the key properties of the model $\Gamma$--limit groups which we construct in this paper is that they are toral relatively hyperbolic, which means that there are many available algorithmic tools which can be applied to them (see, for example 
\cite{Dah-FindRH,dahmanigroves2,dahmanigroves1,dahmaniguirardel:iso}).   In this section, we recall the basic definitions, and some results which we need in order to prove that model $\Gamma$--limit groups are toral relatively hyperbolic.

\begin{definition}\label{d:TRH}
A group is called \emph{toral relatively hyperbolic} if it is torsion-free and hyperbolic relative to finitely generated abelian subgroups. There are many different equivalent definitions of relative hyperbolicity; we refer to \cite{HruskaQC} for a summary of many of them, and proofs that they are equivalent.
\end{definition}

A toral relatively hyperbolic group admits a canonical set of peripheral subgroups, namely a set of representatives for the conjugacy classes of maximal noncyclic abelian subgroups.  We always take this to be the peripheral structure, and so we often leave the peripheral structure unmentioned.

We will construct our relatively hyperbolic models by gluing.  Combination theorems for relatively hyperbolic groups were proved by Alibegovi\'c \cite{alibegovic} and Dahmani \cite{Dahmani03}, and subsequently by Mj--Reeves \cite{MjReeves}, and others.  We state Dahmani's combination theorem here.

\begin{theorem} \cite[Theorem 0.1]{Dahmani03} \label{t:DahComb}
\begin{enumerate}
\item Let $\Gamma$ be the fundamental group of a finite acylindrical graph of relatively hyperbolic groups,
whose edge groups are fully quasi-convex subgroups of the adjacent vertex groups.  Let $\mathcal G$ be the
collection of images of the maximal parabolic subgroups of the vertex groups in $\Gamma$, and their conjugates.
Then $(\Gamma, {\mathcal G})$ is a relatively hyperbolic group.
\item Let $G$ be a group which is hyperbolic relative to $\mathcal G$, and let $P \in \mathcal G$.  Let $A$ be a finitely generated group which has $P$ as a subgroup.  Then $\Gamma = A \ast_P G$ is hyperbolic relative to $(\mathcal H \cup \mathcal A)$, where $\mathcal H$ is the set of $\Gamma$--conjugates of elements of $\mathcal G \smallsetminus P^G$ and $\mathcal A$ is the set of $\Gamma$--conjugates of $A$.
\item Let $G_1$ and $G_2$ be relatively hyperbolic groups, let $P$ be a maximal parabolic subgroup of $G_1$ and suppose that $P$ can be embedded in a parabolic subgroup of $G_2$.  Then $\Gamma = G_1 \ast_P G_2$ is relatively hyperbolic.
\item Let $G$ be a relatively hyperbolic group and let $P, P'$ be non-conjugate isomorphic parabolic subgroups of
$G$.  Then $\Gamma = G \ast_P$ is relatively hyperbolic.
\end{enumerate}
\end{theorem}

The following converse to the combination theorem will also be useful. 

\begin{proposition}\label{p:Vertex groups RH} \cite[Proposition 3.4]{GuiLev:Splittings15}
Suppose that $G$ is a toral relatively hyperbolic group and that $\Lambda$ is a graph of groups decomposition in which the edge groups are abelian and all noncyclic abelian subgroups are elliptic.

Then the vertex groups of $\Lambda$ are relatively quasiconvex, and hence toral relatively hyperbolic.
\end{proposition}

\subsection{Strong accessibility} \label{ss:SA}

Having decomposed a (relatively hyperbolic) group $G$ into the freely indecomposable factors $G_i$ of its Grushko decomposition, and having computed the JSJ decompositions $\Delta_i$ of those factors, one has in a sense reduced the study of $G$ to the study of the rigid vertex groups $G_{i,j}$ of the $\Delta_i$.  It is now natural to apply this procedure again to each $G_{i,j}$, and to continue recursively.  The assertion that this process terminates is called \emph{strong accessibility}. 

We now make this precise.  To a toral relatively hyperbolic group $G$ and a family of subgroups $\mathcal{P}$, we associate a tree $\mathcal{T}_{G,\mathcal{P}}$ of subgroups as follows.  

\begin{enumerate}
\item We start by labelling a node by $(G,\mathcal{P})$.
\item For a node $v$ labelled by a relatively freely decomposable pair $(H,\mathcal{Q})$, we compute the Grushko decomposition $H_1*\ldots*H_k*\mathbb{F}$ of $H$ (with the $H_i$ well defined subgroups up to conjugacy) relative to $\mathcal{Q}$; if $k=1$ and $\mathbb{F} = \{ 1 \}$ we proceed to the next step; otherwise, we attach $k$ edges to $v$, and label the other ends of the edges by the $(H_i,\mathcal{Q}_i)$, where $\mathcal{Q}_i$ consists of all the elements of $\mathcal{Q}$ contained in $H_i$.
\item For a node $v$ labelled by a relatively freely indecomposable subgroup $(H,\mathcal{Q})$, we compute the primary JSJ decomposition $\Delta$ of $H$ relative to $\mathcal{Q}$.
\begin{enumerate}
\item If $\Delta$ is trivial, then add no edges to $v$.
\item If $\Delta$ is non-trivial, and $H_1,\ldots,H_k$ are the rigid vertices of $\Delta$ (again, well defined up to conjugacy), then attach $k$ edges to $v$, and label the other ends of the edges by the $(H_i,\mathcal{Q}_i)$, where $\mathcal{Q}_i$ consists of all the elements of $\mathcal{Q}$ contained in $H_i$.
\end{enumerate}
\item Apply steps 2 and 3 to any remaining nodes.
\end{enumerate}

The Strong Accessibility theorem \cite[Theorem 2.5]{LouderTouikan}  now asserts that this construction terminates after finitely many steps.

\begin{theorem}[Strong Accessibility, Louder--Touikan]\label{t: Strong accessibility}
If $G$ is a toral relatively hyperbolic group and $\mathcal{P}$ is any family of subgroups then the tree $\mathcal{T}_{G,\mathcal{P}}$ is finite.
\end{theorem}

We apply this to deduce that any subgroup of a toral relatively hyperbolic group has a \emph{quasi-convex enclosure}.  Let $H$ be a non-abelian subgroup of a relatively hyperbolic $G$ and $\mathcal{D}$ a finite collection of subgroups of $H$. We suppose that $H$ does not admit a { nontrivial primary} splitting relative to $\mathcal{D}$.   We may now use the hierarchy $\mathcal{T}_{G,\mathcal{D}}$ to define a sequence of subgroups
\[
H\subseteq G_n\subseteq G_{n-1}\subseteq \ldots \subseteq G_0=G
\]
and, for each $i$, an $G_i$--tree $T_i$.  Indeed, we set $v_n$ to be the unique leaf of $\mathcal{T}_{G,\mathcal{D}}$ that contains (a conjugate of) $H$ and let the embedded path from $v_n$ to the root of $\mathcal{T}_{G,\mathcal{D}}$ be $v_n,\ldots,v_0$.  For each $i$, let $G_i$ be such that the label of $v_i$ is $(G_i,\mathcal{D})$ (where if necessary we replace $G_i$ by a suitable conjugate so that $H\subseteq G_i$) and let $T_i$ be the Bass--Serre tree of the splitting of $G_i$ given in the definition of $\mathcal{T}_{G,\mathcal{D}}$ (taking $T_n$ to be a point).  Note that each $T_i$ has a unique vertex $u_i$ stabilized by $H$, and $G_{i+1}$ is the stabilizer of $u_i$.

\begin{definition} \label{d:Enclose}
With the above notation, the \emph{quasiconvex enclosure} of $H$ (relative to $\mc{D}$) is (with a slight abuse of notation) denoted by $\overline{H}$ and defined to be the subgroup $G_n$.
\end{definition}

While this paper was in preparation, the authors discovered that a similar definition is made in \cite[$\S\S 3.4$]{kharlampovich_decidability_2013}, where a ``quasi-convex closure'' is defined similarly, although that definition is not relative to a set of subgroups $\mc{D}$.   The application of the definition in \cite{kharlampovich_decidability_2013} is similar in spirit to the one in this paper, although, as we explain in $\S$\ref{s:KM} above, considerable care is needed to ensure that the groups constructed really do have the good geometric features one would like, and also that simultaneously one does not lose control over the sets of homomorphisms one considers.

\begin{corollary}  \label{c:Enclose}
Suppose that $G$ is a toral relatively hyperbolic group, that $H$ is a finitely generated non-abelian subgroup, and $\mc{D}$ a finite collection of subgroups of $H$.  Suppose that $H$ admits no abelian splitting relative to $\mathcal{D}$.  Then:
\begin{enumerate}
\item the quasiconvex enclosure $\overline{H}$ is the maximal subgroup of $G$ that contains $H$ and does not admit any essential abelian splitting relative to $\mc{D}$; and
\item $\overline{H}$ is relatively quasiconvex (and hence itself toral relatively hyperbolic).
\end{enumerate}
\end{corollary}
\begin{proof}
The first assertion is immediate from the definition.  The second follows from Proposition \ref{p:Vertex groups RH} and induction.
\end{proof}

\subsection{Basics on algorithms with toral relatively hyperbolic groups}

\begin{proposition} \cite[Theorem 0.2]{Dah-FindRH}
There is a one-sided algorithm which takes as input finite presentations $\mc{P}$ and terminates if and only if the group defined by $\mc{P}$ is toral relatively hyperbolic.  In case it is toral relatively hyperbolic, the algorithm provides a basis for each peripheral subgroup.
\end{proposition}

Suppose that we are given a finite presentation of a toral relatively hyperbolic group $\Upsilon$, which is witnessed to be such by the above algorithm.  Then we have an explicit solution to the word problem for $\Upsilon$.   One way to see this is \cite[Theorem 0.1]{dahmani09}, which proves that the existential theory of $\Upsilon$ is decidable.

The next theorem is not immediately implied by the statement of  \cite[Theorem 3.5]{GrovesWilton09}.  However the proof there applies verbatim to prove the following stronger result.
\begin{theorem} [cf. Theorem 3.5, \cite{GrovesWilton09}]\label{t:Compute Z(g)}
There exists an algorithm that, given as input a presentation for a toral relatively hyperbolic group $\Gamma$, and an element $\gamma \in \Gamma$, outputs a minimal set of generators for $Z_\Gamma(\gamma)$.
\end{theorem}
We remark that in case $\Gamma$ is toral relatively hyperbolic, $Z_\Gamma(\gamma)$ is a finitely generated free abelian group.  Therefore the cardinality of a minimal generating set determines the group up to isomorphism.  In particular, it tells us the rank and a presentation can be easily (and algorithmically) written down for $Z_\Gamma(\gamma)$.

\section{Model $\Gamma$--limit groups} \label{s:Defn models}
In this section we make the required definitions of model $\Gamma$--limit groups, which are the central object of concern in this paper.

We start by introducing the main category we work with.  Recall the definition of a {\em non-degenerate} map from Definition \ref{d:degenerate}.

\begin{definition} \label{d:Cat CSA}
The category $\catCSA$ is defined as follows:

$\bullet$  The objects of $\catCSA$ are pairs $(G,\mc{P})$, where $G$ is a torsion-free CSA group and $\mc{P}$ is a finite collection of non-trivial abelian subgroups of $G$.  Denote the set of objects by $\Obj(\catCSA)$.

$\bullet$ The morphisms from $(G,\mc{P})$ to $(H,\mc{Q})$ are the homomorphisms  $\phi \co G \to H$ that are non-degenerate with respect to $\mc{P}$ and such that, for each $P\in\mc{P}$, $\phi(P)$ is conjugate into an element of $\mc{Q}$.  If $X_1, X_2 \in \Obj(\catCSA)$ then $\Mor_{\catCSA}(X_1,X_2)$ denotes the set of morphisms between $X_1$ and $X_2$.
\end{definition}

\begin{lemma}
$\catCSA$ is a category.
\end{lemma}
\begin{proof}
The only thing that needs to be checked is that morphisms compose. They do, since the distinguished subgroups of one group map into the distinguished subgroups of another, so the composite homomorphism remains injective on distinguished subgroups.
\end{proof}

The construction of model $\Gamma$--limit groups we give in the next section relies on colimits in $\catCSA$.

In this section we give the definition of models in terms of a universal property, leaving open the question of existence.  In Section \ref{s:Gluing} we provide the construction of models.

Let $L_0$ be a $\Gamma$--limit group and let 
\[	L_0 \stackrel{\lambda_0}{\to} L_1 \stackrel{\lambda_1}{\to} \cdots \stackrel{\lambda_{n-1}}{\to} L_n	\]
be a strict resolution of $L_0$, where $L_n$ admits a (given) strict map $\lambda_n \co L_n \to \Gamma$.  Call this resolution (along with the map $\lambda_n$) $\res_0$.  For $i \ge 0$, let $\res_i$ denote the resolution obtained from $\res_0$ by starting at $L_i$, and finally appending the map $\lambda_n \co L_n \to \Gamma$. 

Below we define the {\em model $\Gamma$--limit groups of $L_i$ with respect to $\res_i$}, denoted $\model_{\res_i}$. The definition is made by (reverse) induction on $i$.  

Before we describe the construction, we note a few important properties of models, which will be proved throughout the subsequent sections.

\begin{theorem} \label{t:prelim properties of models} Let $\res_0$ be as above, and let $\mathcal{C}_i$ be the edge groups of the JSJ decomposition of the freely indecomposable factors of the Grushko decomposition of $L_i$. For $i \ge 0$, there exists a canonical group $\model_{\res_i}$, and these groups satisfy the following properties:
\ 
\begin{enumerate}
\item $\model_{\res_i}$ is a toral relatively hyperbolic $\Gamma$--limit group;
\item There is a strict map $\mu_i \co \model_{\res_i} \to \model_{\res_{i+1}}$;
\item There is a distinguished family ${\mc{D}_{\res_i}}$ of abelian subgroups of $\model_{\res_i}$ so that $(\model_{\res_i},\mc{D}_{\res_i}) \in \Obj(\catCSA)$;
\item There are injective maps $\eta_i \co L_i \to \model_{\res_i}$ so that $\mu_i \circ \eta_i = \eta_{i+1} \circ \lambda_i$;
\item The map $\eta_i$ takes $\mathcal{C}_i$ to $\mc{D}_{\res_i}$ and is non-degenerate with respect to $\mc{C}_i$.  In other words, $\eta_i \in \Mor(\catCSA)$.
\end{enumerate}
\end{theorem}
The claims of Theorem \ref{t:prelim properties of models} are contained in Theorems \ref{t:Model props} and \ref{t:Properties of models} below.

We now give the definition of models.
First suppose that $L_0$ is freely indecomposable.  

Suppose inductively that we have defined $\model_{\res_1}$, along with a homomorphism $\eta_1 \co L_1 \to \model_{\res_1}$.
Consider the primary JSJ decomposition $\Delta$ of $L_0$, and let $\{ \Delta_v \}$ be the collection of rigid vertex groups.  Let $\Delta_v$ be such a rigid vertex group.  Let $\mathcal C_v = \{ \Delta_{e_1}, \ldots , \Delta_{e_s} \}$ be the set of adjacent edge groups.  We remark that these edge groups are maximal abelian subgroups of $\Delta_v$.

We define the {\em relative model} of $\Delta_v$, denoted $\rmodel(\Delta_v)$, to be the maximal subgroup of $\model_{\res_1}$ which does not admit any essential abelian splitting relative to $(\eta_1 \circ \lambda_0) (\mathcal C_v)$.  If we suppose by induction that $\model_{\res_1}$ is a toral relatively hyperbolic group,  such a maximal subgroup exists by Corollary \ref{c:Enclose}, and indeed is the quasiconvex enclosure $\overline{\eta_1 \circ \lambda_0(\Delta_v)}$ (relative to $(\eta_1 \circ \lambda_0) (\mathcal C_v)$).  In particular, $\rmodel(\Delta_v)$ is toral relatively hyperbolic, and hence finitely presented, torsion-free and CSA by that same result.

Consider the elements $\left\{ (\Delta_{v},\mc{C}_{v}) \right\}$, $\left\{ ( \rmodel(\Delta_{v}), \overline{\mc{C}}_v) \right\}$ of $\Obj(\catCSA)$, where $\overline{\mc{C}}_v$ is the set $\{ (\eta_1\circ\lambda_0) (C) \mid C \in \mc{C}_v\}$.  Also consider the set $\mc{C}$ of all peripheral subgroups of edge groups of the JSJ of $L_0$, and the element $( L_0, \mc{C}) \in \Obj(\catCSA)$.

The model $\model_{\res_0}$ is then the group that appears in the colimit of the following diagram (labelling only groups and not the attendant subgroups), in the category $\catCSA$ (supposing for the moment that such a colimit exists):

\medskip
\
\centerline{
\xymatrix{
    \Delta_{v_1} \ar@{>}[dd]^{\eta_1 \circ \lambda_0} \ar@{>}[rr]^{\iota} \ar@{.}[dr] & &  L_0 \\
 & \Delta_{v_n} \ar@{>}[ur]^{\iota} \ar@{>}[dd]^{\eta_1 \circ \lambda_0} & \\
 \rmodel(\Delta_{v_1}) \ar@{.}[dr] & & \\
 & \model(\Delta_{v_n}) & \\
}}
\medskip

The map $\eta_0 \co L_0 \to \model_{\res_0}$ is the one furnished by the definition of colimit.  Further,  the groups $\rmodel(\Delta_{v_i})$ are subgroups of $\model_{\res_1}$, and we have a map $\eta_1 \circ \lambda_0 \co L_0 \to \model_{\res_1}$.  Temporarily define a set of maximal abelian subgroups $\mc{D}$ of $\model_{\res_1}$ by taking images in $\model_{\res_1}$ of $\mc{C}$ and of the $\overline{\mc{C}}_v$.  Note that this will typically not be the set of maximal abelian subgroups of $\model_{\res_1}$ that it was furnished with when {\em it} was defined as a pushout in $\catCSA$.

  Since the maps $\eta_1\circ \lambda_0$ and inclusion of $\rmodel(\Delta_{v_i})$ into $\model_{\res_1}$ agree on each $\Delta_{v_i}$, and these maps are all in $\Mor(\catCSA)$, the universal property of the colimit yields a map $\mu_0 \co \model_{\res_0} \to \model_{\res_1}$ so that
$\mu_0 \circ \eta_0 = \eta_{1} \circ \lambda_0$.  Note also that (so long as it exists) the universal property of $\model_{\res_0}$, together with the distinguished subgroups of the groups $\Delta_v$, $\rmodel(\Delta_v)$ and $L_0$, furnish $\model_{\res_0}$ with a collection of distinguished subgroups.

Now suppose that $L_0$ is freely decomposable
\[	L_0 = H_1 \ast \ldots \ast H_r \ast F_k	\]
where each $H_i$ is freely indecomposable, and $F_k$ is a free group.  Let $\res_{0,i}$ denote the resolution obtained by restricting $\res_0$ to the subgroup $H_i$ of $L_0$.

   The definition above gives the models $\model_{\res_{0,i}}$ along with maps to the model of $L_1$.
The model of $L_0$ is
\[	\model_{\res_0} = \model_{\res_{0,1}} \ast \ldots \ast \model_{\res_{0,r}} \ast F_k	\]
and the map $\eta_0$ is defined in the obvious way.

\begin{remark}
At the moment, it is not obvious that the colimit above exists.  If it does, it is not clear that the model should be a $\Gamma$--limit group, or what its properties should be (except that if it exists it is clearly torsion-free and CSA).  Much of the rest of the paper is spent proving that models exist and understanding the properties of models, in order to see what they are and why they are useful.
\end{remark}

If we apply the definition of model to the strict resolution $\Gamma \to \Gamma$ given by the identity map, we can see that with respect to any collection $\mc{P}$ of abelian subgroups, the model of $(\Gamma,\mc{P})$ is $(\Gamma,\mc{P})$.

\begin{proposition} \label{p:res of length 0}
For any collection of nontrivial abelian subgroups $\mc{P}$ of $\Gamma$,
the model of $(\Gamma,\mc{P})$ (with respect to the trivial resolution) is $(\Gamma,\mc{P})$.
\end{proposition}
\begin{proof}
In case $\Gamma$ is freely indecomposable, this is immediate from the observation that $\rmodel(\Delta_v)= \Delta_v$.
The freely decomposable case is also immediate.
\end{proof}

\section{Abelian graphs of groups in $\catCSA$ and the construction of models}
\label{s:Gluing}

Let $L_0$ be a freely indecomposable $\Gamma$--limit group, equipped with a strict resolution $\res_0$.
 In this section, we give a construction of $\model_{\res_0}$ which is amenable to understanding the structure of the model of $L_0$.

Let $\catCSA$ be the category defined in Definition \ref{d:Cat CSA} above.
In order to prove our structural results about the model of $L_0$, we investigate certain kinds of colimits in $\catCSA$.  
Along the way, we exhibit various subtleties about  building torsion-free CSA groups by taking graphs of groups.  We will then show (in case $L_0$ is freely indecomposable) that we can construct $\model_{\res_0}$ as a graph of groups with abelian edge groups.

\subsection{Some colimits}

The following simple example illustrates that the class of torsion-free CSA groups is not closed
under taking pushouts in the category of groups.
\begin{example} \label{ex:Not Comm Trans}
Let $A_1 = A_2 = C = \Z$ and let $\iota_i \co C \to A_i$ be index $2$ embeddings for $i =1,2$.  Then the pushout in the category of groups is $A_1 \ast_C A_2 = \Z \ast_{2\Z} \Z$ which is not commutative transitive.

In the category of abelian groups, the pushout is $\Z \times \Z/2\Z$, which is not torsion-free.

In the category of free abelian {groups}, the pushout is $\Z$, with the maps from the $A_i$ being the same isomorphism and the map from $C$ being the index $2$ embedding.
\end{example}
The next example shows that we do not want to merely take pushouts in the category of torsion-free CSA groups, where all homomorphisms are allowed, but required our maps to be non-degenerate.

\begin{example}
Let $A = \left\langle a_1,a_2 \mid [a_1,a_2] \right\rangle \cong \Z^2$, $B = \langle b_1,b_2 \mid [b_1,b_2] \rangle \cong \Z^2$ and $C = \langle c \mid \varnothing \rangle \cong \Z$.  Let $\iota_1 \co C \to A$ be defined by $\iota_1(c) = a_1$ and $\iota_2 \co C \to B$ by $\iota_2(c) = b_1$.  Then the pushout in the category of groups is 
$G = \langle x,y,z \mid [x,y], [y,z] \rangle \cong {\mathbb F}_2 \times \Z$, which is not CSA.

Certainly, $G_{ab} = \Z^3$ is a torsion-free CSA quotient of $G$.  If we added peripheral structures $\mc{P}_C = \{ C \}$ and the images of $C$ in $A$ and $B$ to $A$ and $B$, then this would be the pushout in $\catCSA$.

On the other hand, killing $y$ gives ${\mathbb F}_2 = \langle x,z \mid \varnothing \rangle$ as a torsion-free CSA quotient of $G$.  These are incomparable quotients, in the sense that there is no torsion-free CSA quotient of $G$ which `lies above' them both.\footnote{Of course, this phenomenon is well known from the construction of Makanin--Razborov diagrams since $G_{ab}$ and $\mathbb F_2$, equipped with the canonical quotient maps from $G$ form the first level of the MR diagram for $G$ (over a nonabelian free group $\mathbb F$).  It is also one of the key reasons for the distinguished abelian subgroups and the requirement that homomorphisms be non-degenerate.}
\end{example}

To remedy these difficulties, we consider the pushout in the category $\catCSA$, and in particular we restrict to non-degenerate maps. 

First, we explicitly describe the pushout in the category of finitely generated abelian groups.  It is easy to see that pushouts exist in the category of abelian groups (the pushout is the abelianization of the pushout in the category of groups).  However, we are interested in free abelian groups rather than merely abelian groups.

\begin{lemma}
Let $A_1,A_2,C$ be finitely generated nontrivial free abelian groups, and let $\iota_i \co C \to A_i$ be embeddings, for $i =1,2$.  There exists a pushout of the diagram

\centerline{
\xymatrix{
   C  \ar@{>}[d] \ar@{>}[r] & A_1 \\
 A_2 }}

\noindent in the category of free abelian groups.
\end{lemma}
\begin{proof}
This is straightforward.  One way to proceed is to pass to direct factors of the $A_i$ which contain the image of $C$ as finite-index subgroups, embed as lattices in $\R^k$, and take the span of these two lattices.
\end{proof}

The following lemma is straightforward.
\begin{lemma} \label{l:restricted pushout}
If $M$ is the pushout of a collection of finitely generated free abelian groups $\{ A_i \}$ (with appropriate maps), and $H$ is a finitely generated free abelian group with injective maps $A_i \into H$ for each $i$, making the appropriate diagram commute, then the induced map from $M$ to $H$ is injective.
\end{lemma}

This is useful because of the following.

\begin{corollary}
If all groups are free abelian and all maps are injective then
 the pushout of the diagram

\centerline{
\xymatrix{
   C  \ar@{>}[d] \ar@{>}[r] & A \\
 B & }}
 \noindent
in the category of finitely generated free abelian groups is also the pushout of that diagram in $\catCSA$ 
so long as the set of distinguished subgroups of $C$ is nonempty, and each map is a morphism in $\catCSA$.
\end{corollary}

In Subsection \ref{s:Construction} below, we will need a slightly more general construction than a colimit, involving a graph of groups.

\subsection{Construction of models} \label{s:Construction}

\begin{definition}
Suppose that $G$ is a finitely generated, torsion-free CSA group.  A graph of groups decomposition $\Delta$ of $G$ is {\em JSJ--like} if it satisfies the following properties:
\begin{enumerate}
{ \item $\Delta$ is universally elliptic, i.e.\ the edge stabilizers of $\Delta$ are elliptic in every primary splitting of $G$.}
\item Every non-cyclic abelian subgroup of $L_0$ is elliptic in $\Delta$.
\item  $\Delta$ is bipartite, with one class of vertex stabilizers abelian and the other either (nonabelian) rigid { (in the sense of Definition \ref{d:rigid})} or a maximal socket (with every boundary component used).  We denote the set of abelian vertices by $\{w\}$, the rigid vertices by $\{u\}$ and the socket vertices by $\{v\}$.
\item Every edge group is maximal abelian in the adjacent vertex group which is not abelian.  Furthermore, $\Delta$ is $2$--acylindrical.
\end{enumerate}
\end{definition}
Note that in case $G$ is freely indecomposable, { it follows immediately from Proposition \ref{p:JSJ props} that the canonical primary JSJ decomposition of $G$ is JSJ--like, and that every JSJ--like decomposition is in fact a canonical primary JSJ.  We introduce this definition for the slight extra flexibility of applying it in case that we do not know that $G$ is freely indecomposable.}

Suppose that $L$ is a freely indecomposable $\Gamma$--limit group, that $\Delta$ is a JSJ--like decomposition of $L$ and that
$\rho \co L \to M$ is a strict map {with respect to $\Delta$}, where $M$ is a toral relatively hyperbolic group.  In our applications, $M$ will be the model of $L_1$, where $L \to L_1$ is the beginning of a strict resolution of $L$.

Given $\Delta$ and $\rho$, we build a new graph of groups $\wh\Delta$ as follows:

For a rigid vertex $u$ of $\Delta$, we consider the map $\rho |_{\Delta_u} \co \Delta_u \to M$.  Since $\lambda_0$ is strict, $\rho |_{\Delta_u}$ is injective.  We define
$\wh{\Delta}_u$ to be the quasiconvex enclosure of $\rho(\Delta_u)$ in $M$ {(relative to the images under $\rho$ of the adjacent edges groups)}.  Let $\eta_{u} = \rho |_{\Delta_u}$ be the natural inclusion from $\Delta_u$ to $\wh{\Delta}_u$.

 Define an equivalence relation on the edges of $\Delta$ as follows:  Two edges $e_1$ and $e_2$ adjacent to a rigid vertex $u$ are equivalent if $Z_{\wh{\Delta}_u}(\eta_u(\Delta_{e_1}))$ and $Z_{\wh{\Delta}_u}(\eta_u(\Delta_{e_2}))$ are conjugate in $\wh{\Delta}_u$.  Edges not adjacent to rigid vertices are equivalent only to themselves.

Suppose that $e_1$ and $e_2$ are adjacent to a nonabelian vertex $y$ and are equivalent. Let $Z_1 = Z_{\wh{\Delta}_y}(\eta_y(\Delta_{e_1}))$ and $Z_2 = Z_{\wh{\Delta}_y}(\eta_y(\Delta_{e_2}))$.  By CSA, {\em any} $\gamma \in \wh{\Delta}_y$ which conjugates $Z_1$ to $Z_2$ induces the same isomorphism, so there is a canonical isomorphism between $Z_1$ and $Z_2$.

  The underlying graph of $\widehat{\Delta}$ is equal to the quotient graph of the underlying graph of $\Delta$ obtained by identifying edges which are equivalent according to the above relation.  Note that there is a natural way of labeling each vertex of this graph as exactly one of abelian, rigid or socket.  If $y$ is a vertex of $\Delta$ we let $[y]$ denote the corresponding vertex in $\wh{\Delta}$ under the quotient map; likewise, the edges of $\wh{\Delta}$ are the equivalence classes $[e]$ where $e$ is an edge of $\Delta$.  Note, however, that the equivalence classes $[u]$ of rigid vertices and $[v]$ of socket vertices are singletons, and so we will usually write $u$ for $[u]$ and $v$ for $[v]$.
  
We now define the edge and vertex groups of $\wh{\Delta}$ as follows:

\begin{enumerate}
\item A rigid vertex $u$ is labelled by the group $\wh{\Delta}_{u}$ defined above.
\item A socket vertex $v$ is labelled by the socket group $\wh{\Delta}_{v}$, obtained from $\Delta_v$ by keeping the same surface, but adding any extra roots that exist in $M$ to the boundary components.
\item Let $[e]$ be an edge joining an abelian vertex $[w]$ to a vertex $y$.  Note that $y$ is not an abelian vertex.
Then $\wh{\Delta}_{[e]} = Z_{\wh{\Delta}_y}(\eta_y(\Delta_e))$.
\item Finally, let $[w]$ be an abelian vertex, and let $[e_1], \ldots , [e_l]$ be the adjacent edges.  If $[w]=\{w_k\}$ and $[e_i]=\{e_{i,j}\}$ then we have two sets of inclusions of free abelian groups:
\begin{enumerate}
\item $\iota_{i,j}\co\Delta_{e_{i,j}}\to\wh{\Delta}_{[e_i]}$ for each $i,j$;
\item $\kappa_{i,j}\co\Delta_{e_{i,j}}\to\Delta_{w_k}$ whenever $e_{i,j}$ is incident at $w_k$.
\end{enumerate}
We let these data define a graph of groups $\Lambda_{[w]}$ in the natural way: the underlying graph is bipartite, with vertex set the disjoint union of the $\{w_k\}$ and $\{[e_i]\}$, and edge set $\{e_{i,j}\}$ equipped with the natural incidence relations; the vertex group of $w_k$ is $\Delta_{w_k}$, the vertex group of $[e_i]$ is $\wh{\Delta}_{[e_i]}$,  the edge group of $e_{i,j}$ is $\Delta_{e_{i,j}}$, and the edge maps are given by $\iota_{i,j}$ and $\kappa_{i,j}$.   The vertex group $\wh{\Delta}_{[w]}$ is now defined to be the abelianization of $\pi_1\Lambda_{[w]}$.
\end{enumerate}

We remark that the definition of $\wh{\Delta}_{[e]}$ used the representative $e$ of $[e]$; however we have already noted that the appropriate edge groups are {\em canonically} isomorphic, so there is no ambiguity in this definition.  

Finally, the edge maps of $\wh{\Delta}$ are defined in the natural way using the canonical isomorphisms of conjugacy representatives of edge groups and either inclusion (in the case of rigid and socket vertices) or the natural inclusion of a vertex group in the fundamental group of a graph of groups (in the abelian case).

We call the graph of groups $\wh{\Delta}$ the {\em expansion} of $\Delta$ with respect to $\rho \co L \to M$.

\begin{lemma} \label{l:Model is TRH}
Suppose that $L$ is a freely indecomposable $\Gamma$--limit group, and that $\Delta$ is a JSJ--like decomposition of $L$.  Suppose further that $\rho \co L \to M$ is a strict map, that $M$ is toral relatively hyperbolic, and that $\wh\Delta$ is the expansion of $\Delta$ with respect to $\rho$.  Then $\pi_1(\wh\Delta)$ is toral relatively hyperbolic, and in particular torsion-free and CSA.
\end{lemma}
\begin{proof}

By construction, the edge groups incident at each rigid vertex form a malnormal family.  In particular, we may apply Dahmani's Combination Theorem \ref{t:DahComb} together with the fact that the vertex groups $\wh{\Delta}_v$ are toral relatively hyperbolic (by Corollary \ref{c:Enclose}) to deduce that $\pi_1\wh{\Delta}$ is toral relatively hyperbolic, in particular torsion-free and CSA.
\end{proof}

There is a natural homomorphism $\eta \co L \to \pi_1(\wh{\Delta})$ defined as follows.  For rigid vertices $\Delta_u$, the map $\eta$ is given by $\eta_u \co \Delta_u \to \wh{\Delta}_u$.  For socket vertex groups $\Delta_v$, $\eta$ is defined to be the inclusion of the socket $\Delta_v$ into the socket $\wh{\Delta}_v$.  For abelian vertex group $\Delta_w$, note that $\Delta_w$ is a vertex group of $\Lambda_{[w]}$, and so $\eta$ is defined to be the composition of the inclusion map $\Delta_w \to \pi_1\Lambda_{[w]}$ with abelianization.  Finally, after picking a maximal tree in the underlying graph of $\Delta$, each stable letter either maps to a stable letter of $\wh{\Delta}$ or to an {element} of $\pi_1\Lambda_{[w]}$ for some abelian vertex group $w$.

\begin{proposition} \label{p:Map is injection}
In the above situation $\eta \co L \to \pi_1\wh{\Delta}$ is an injection.
\end{proposition}
\begin{proof}
In order to prove injectivity of the map $\eta$ it is convenient to modify $\Delta$ to produce a new graph of groups $\Delta'$.  The underlying graph of $\Delta'$ is the same as the underlying graph of $\Delta$.  If ${v}$ is a { (flexible)} socket vertex of $\Delta$ then we take $\Delta'_{v}=\Delta_{v}$, and likewise if $w$ is an abelian vertex we take $\Delta'_w=\Delta_w$.   For ${u}$ a rigid vertex of $\Delta$, we set $\Delta'_{u}=E(\Delta_{u})\cap \rho^{-1}({\widehat{\Delta}}_{u})$ (recalling from Definition \ref{d:Envelope} that $E(\Delta_{u})$ denotes the envelope of $\Delta_{u}$).  Finally, for $e$ an edge adjoining an abelian vertex $w$ and a socket or rigid vertex $y$, we take $\Delta'_e=\Delta'_w\cap \Delta'_y$.  We make three remarks about the construction of $\Delta'$.

First, $\Delta'$ was obtained from $\Delta$ by pulling certain elements across edges, and so $\pi_1\Delta'= L$. 

Second, if $y$ is a rigid or socket vertex in $\Delta$, then it remains the case that the incident edge groups in $\Delta'$ form a malnormal family in $\Delta'_y$; in particular, $\Delta'$ is also $2$--acylindrical.

Third, for each vertex or edge $y$ of $\Delta$, we have that $\eta(\Delta'_y)\subseteq \widehat{\Delta}_{[y]}$, by the definition of $\Delta'_y$ in the different cases.  It follows that there is a morphism of trees $\alpha\co T'\to\widehat{T}$ (the Bass--Serre trees of $\Delta'$ and $\widehat{\Delta}$ respectively) that intertwines $\eta$ -- that is, $\alpha(g.x)=\eta(g)\alpha(x)$ for any $g\in L$ and $x\in T'$.  

To show that $\eta$ is injective, it now suffices to prove that $\eta$ is injective on vertex groups of $\Delta'$ (which follows from strictness of $\rho$) and that $\alpha$ is a local injection (i.e. does not factor through a fold).  
In order to prove that $\alpha$ does not factor through a fold, we must prove that for each edge $e$ and each vertex $y$ we have $\eta^{-1}(\widehat{\Delta}_{[e]}) \cap \Delta'_y = \Delta'_e$.  Of course, it is clear that $\Delta'_e \subseteq \eta^{-1}(\widehat{\Delta}_{[e]}) \cap \Delta'_y$, and it is the reverse inclusion that must be proved.

So, suppose that $g \in \eta^{-1}(\widehat{\Delta}_{[e]}) \cap \Delta'_y$.  As usual, there are three cases, depending on the type of $y$.  In case $y$ is a socket vertex,  the only difference between $\Delta'_y$ and $\wh{\Delta}_y$ is that there may be extra roots of boundary components of the surface which exist in $M$ but not in $L$.  Thus in this case we clearly have $g \in \Delta{'}_e$, as required.  Suppose then that $y$ is a rigid vertex group.  By definition $\Delta'_y =E(\Delta_y)\cap \rho^{-1}({\widehat{\Delta}}_y)$ and $\Delta'_e =\Delta'_w\cap \Delta'_y$, where $w$ is the vertex at the other end of $e$.
Also, by definition, $\widehat{\Delta}_{[e]} = Z_{\widehat{\Delta}_y}({\eta_y}(\Delta_e))$ (for a suitable choice of representative edge $e$).  But $g \in \Delta'_y$, so $g \in Z_{\Delta'_y}(\Delta_e) = \Delta'_e$, by construction and malnormality.

Finally, suppose that $y$ is an abelian vertex (so $\Delta'_y=\Delta_y$), and let $x$ be the vertex at the other end of $e$.  If $x$ is a socket vertex then there is nothing to prove, so we assume that $x$ is rigid.

{
The defining graph of groups $\pi_1\Lambda_{[y]}$, of which $\widehat{\Delta}_{[y]}$ is the abelianization, has a quotient obtained from killing each peripheral subgroup of each $\Delta_{y_i}$, and also each vertex group of the form $\widehat{\Delta}_{[e_i]}$. By abelianizing, we obtain a quotient of $\widehat{\Delta}_{[y]}$.  In this quotient, the image of each $\widehat{\Delta}_{[e_i]}$ is trivial, and the image of $\Delta_y$ is $\Delta_y/\overline{P}(\Delta_y)$.  Since $g \in \eta^{-1}(\widehat{\Delta}_{[e]})$, its image in this quotient is trivial.  Therefore, $g \in \overline{P}(\Delta_y) \subseteq E(\Delta_x)$} and so, since $\eta(g)$ is contained in $\widehat{\Delta}_{[e]}\subseteq \widehat{\Delta}_x$, we have
\[
g\in  E(\Delta_x)\cap \eta^{-1}(\widehat{\Delta}_x) =  E(\Delta_x)\cap \rho^{-1}(\overline{\Delta}_x)=\Delta'_x
\]
whence, since $g$ is also contained in $\Delta'_y$, we have that $g\in\Delta'_x\cap\Delta'_y=\Delta'_e$ as required.
\end{proof}

\begin{proposition} \label{p:Is a model}
Suppose that $L_0$ is a freely indecomposable $\Gamma$--limit group and that
$\res_0$ is a strict resolution of $L_0$, starting with the strict map $\lambda_0 \co L_0 \to L_1$ and continuing with the strict resolution $\res_1$ of $L_1$.  Suppose further that $\model_{\res_1}$ exists, and satisfies:
\begin{enumerate}
\item The natural map $\eta_1 \co L_1 \to \model_{\res_1}$ is an injection; and
\item $\model_{\res_1}$ is toral relatively hyperbolic.
\end{enumerate}
Let $\rho \co L_0 \to \model_{\res_1}$ be the map $\rho = \eta_1 \circ \lambda_0$,  let $\Delta$ be a JSJ--like decomposition of $L_0$ and let $\wh{\Delta}$ be the expansion of $\Delta$ with respect to $\rho$.

Equip $\pi_1\wh{\Delta}$ with the distinguished set $\mc{D}_{\wh\Delta}$ of peripheral subgroups of abelian vertex groups of $\wh\Delta$.  Then $(\pi_1\wh{\Delta},\mc{D}_{\wh\Delta})$ is the model of $L_0$ with respect to $\res_0$.
\end{proposition}
\begin{proof}
We briefly recall what it means for $(\pi_1\wh{\Delta},\mc{D}_{\wh\Delta})$ to be the required model and what we have to prove.  Recall that we have, for each rigid vertex $v$ of $\Delta$, two pairs:
\begin{enumerate}
\item $(\Delta_v,\mc{C}_v)$, where $\Delta_v$ is the vertex group, and $\mc{C}_v$ is the family of adjacent edge groups; and 
\item $(\rmodel_{\res_{1}}(\Delta_v), \bar{\mc{C}}_v)$, where $\rmodel_{\res_{1}}(\Delta_v)$ is the relative model (quasiconvex enclosure) of $\Delta_v$ and $\bar{\mc{C}}_v$ is the set of centralizers in $\rmodel_{\res_{1}}(\Delta_v)$ of the images of $\mc{C}_v$.
\end{enumerate}
We also have the pair $(L_0,\mc{D}_\Delta)$, where $\mc{D}_\Delta$ is the set of peripheral subgroups of abelian vertex groups of $\Delta$.

Equip $\pi_1\widehat{\Delta}$ with the collection $\mc{D}_{\widehat{\Delta}}$ of peripheral subgroups of abelian vertex groups of $\widehat{\Delta}$.  We claim that $(\pi_1\widehat{\Delta},\mc{D}_{\widehat{\Delta}})$ is the required colimit in $\catCSA$.  Indeed, suppose that $(H,\mc{Q}_0)$ is an object in $\catCSA$ and $\alpha\in \Mor_{\catCSA}((L_0,\mc{D}_\Delta), (H,\mc{Q}_0))$, $\beta_v\in \Mor_{\catCSA}(({\rmodel_{\res_1}(\Delta_v)}, \bar{\mc{C}}_v), (H,\mc{Q}_0))$ are morphisms that make the diagram below commute.

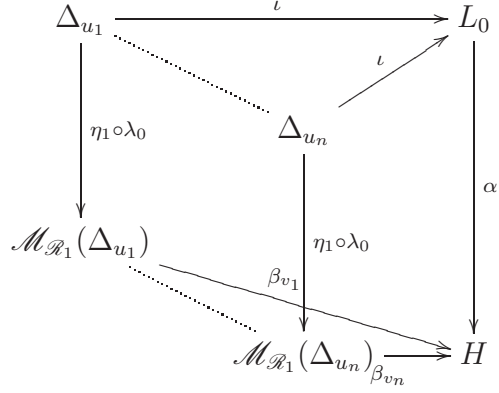
\centerline{
\xymatrix{
    \Delta_{u_1} \ar@{>}[dd]^{\eta_1 \circ \lambda_0} \ar@{>}[rr]^{\iota} \ar@{.}[dr] & &  L_0 \ar[ddd]^\alpha \\
 & \Delta_{u_n} \ar@{>}[ur]^{\iota} \ar@{>}[dd]^{\eta_1 \circ \lambda_0} & \\
 \rmodel_{\res_1}(\Delta_{u_1}) \ar@{.}[dr]  \ar[drr]^{\beta_{v_1}}  & & \\
 & \rmodel_{\res_1}(\Delta_{u_n}) \ar[r]_{\beta_{v_n}} &  H\\
}}

{The map $\pi\co\pi_1\widehat{\Delta}\to H$ is defined as follows.}

\begin{enumerate}
\item For each rigid vertex $u$, $\widehat{\Delta}_{u}=\rmodel_{\res_0}(\Delta_u)$ and $\pi|_{\widehat{\Delta}_{u}}=\beta_{u}$.
\item For each socket vertex $v$, $\widehat{\Delta}_{v}$ { is obtained from attaching extra roots to boundary elements of the socket $\Delta_v$.  This means that $\widehat{\Delta}_v$ is the colimit of a diagram obtained by restricting the maps $\iota$ and $\eta_1\circ\lambda_0$ to subgroups.  The universal property of this colimit defines the map $\pi|_{\widehat{\Delta}_{u_i}}$}
\item Consider an abelian vertex $[w]$. The commutative diagram above shows that the maps $\alpha$ and $\beta_u$ naturally induce a unique, non-degenerate map $\pi_1\Lambda_{[w]}\to H$.  Non-degeneracy, together with the hypothesis that $H$ is CSA, implies that this map factors through the abelianization.
\item We may choose a maximal tree for $\Delta$ that maps to a maximal tree for $\wh{\Delta}$.  Therefore, each stable letter $\wh{t}$ of $\widehat{\Delta}$ corresponds to a stable letter $t$ in $\Delta$ and we define $\pi(\wh{t})=\alpha(t)$. 
\end{enumerate}

It is easy to check that $\pi$ is a homomorphism.  It remains to check that $\pi$ respects the associated families of distinguished subgroups.  But this is clearly the case, because the homomorphisms $\alpha$ and $\beta_v$ do, and the distinguished subgroups of $\pi_1\widehat{\Delta}$ are images of distinguished subgroups of $L_0$ and $\rmodel_{\res_1}(\Delta_v)$.
\end{proof}

We can now prove the following theorem by induction on the length of a strict resolution.

\begin{theorem} \label{t:Model props}
As above, let $L_0$ be a $\Gamma$--limit group with strict resolution $\res_0$.  
\begin{enumerate}
\item $\model_{\res_0}$ exists.
\item The natural map $\eta_0\co L_0\to\model_{\res_0}$ is an injection.
\item $\model_{\res_0}$ is a toral relatively hyperbolic group, in particular is finitely presentable.
\end{enumerate}
\end{theorem}
\begin{proof}
We have already observed in Proposition \ref{p:res of length 0} that, when $\res_0$ has length $0$ and $L_0=\Gamma$, the model for $\Gamma$, with respect to any collection of distinguished subgroups, is itself $\Gamma$.  Therefore, we may assume that $\res_0$ is of length at least one.  We further assume by induction that $\model_{\res_1}$ exists and satisfies the three properties in the statement of the theorem.

Also, it is clear that the theorem is stable under passing to free products, so
we may assume that $L_0$ is freely indecomposable.

Let $\Delta$ be the canonical primary JSJ of $L_0$, and recall that $\Delta$ is JSJ--like.  Let $\lambda_0 \co L_0 \to L_1$ be the first map from $\res_0$ and $\eta_1 \co L_1 \to \model_{\res_1}$ be the natural map from $L_1$ to its model.  By induction, we know that $\eta_1$ is injective, so $\rho = \eta_1 \circ \lambda_0$ is a strict map.  Let $\wh{\Delta}$ be the expansion of $\Delta$ with respect to $\rho$.

It follows from Proposition \ref{p:Is a model} that $\pi_1\wh{\Delta}$ is the model of $L_0$ with respect to $\res_0$.  The other two properties in the statement of the theorem follow immediately from Lemma \ref{l:Model is TRH} and Proposition \ref{p:Map is injection}.
\end{proof}

\section{Further properties of model $\Gamma$--limit groups}
\label{s:Model Properties}

The purpose of this section is to prove further properties of models, which are summarized in the next result.  In order to state the result, we adopt the following notation.  Let  $L$ be a $\Gamma$--limit group and $\res$ a strict resolution of $L$.  Let $\Delta$ be the canonical primary JSJ decomposition of $L$.  Let $\lambda \co L \onto L_1$ be the first map in $\res$.  Let $\model_{\res}$ be the model of $L$ with respect to $\res$, let $\mu \co \model_{\res} \to \model_{\res_1}$ be the associated map between models and let $\eta\co L\hookrightarrow\model_\res$ be the canonical inclusion.  Let $\eta_1 \co L_1 \into \model_{\res_1}$ be the canonical inclusion and let $\widehat{\Delta}$ be the expansion of $\Delta$ with respect to $\eta_1 \circ \lambda$, so that $\model_{\res} = \pi_1 \wh\Delta$.

\begin{theorem}\label{t:Properties of models}
Let $L$ be a freely indecomposable limit group with a strict resolution $\res$ and let $\model_{\res}$ be the associated model for $M$.
\begin{enumerate}
\item $\model_{\res}$ is freely indecomposable. 
\item $\widehat{\Delta}$ is the primary JSJ decomposition for $\model_{\res}$.
\item The  map $\mu \co \model_{\res} \to \model_{\res_1}$ is strict.
\item The model $\model_\res$ is a $\Gamma$--limit group.
\item There is a canonical homomorphism $\Phi\co\Mod(L)\to\Mod(\model_\res)$ that intertwines the injection $\eta$; that is, for $\alpha\in\Mod(L)$ and $g\in L$, we have that $\eta(\alpha(g))=\Phi(\alpha)(\eta(g))$.
\end{enumerate}
\end{theorem}

The proof of this theorem is spread out over a variety of results in this section.  The reader should note the key implication:  the universal property of $\model_\res$ allows one to take a factorization of a homomorphisms $f\co L \to \Gamma$ through $\res$ and induce a factorization of a homomorphism $\wh f \co \model_\res \to \Gamma$ through the induced resolution $\wh\res$ of $\model_\res$.  {\em This is the key point of the construction of models.}

We start with item (1), which is a consequence of a standard result that characterizes when graphs of groups admit free splittings.

\begin{lemma}
With the above notation, $\model_{\res}$ is freely indecomposable.
\end{lemma}
\begin{proof}
For a contradiction, we suppose that $\model_{\res}$ acts on a tree $T$ with trivial edge stabilizers and without a global fixed point.  Let $\wh{\Delta}_1,\ldots,\wh{\Delta}_k$ be the components of $\wh{\Delta}$ created when we delete all non-cyclic abelian edge groups.  For some $i$, $\pi_1\wh{\Delta}_i$ acts on $T$ without a global fixed point.  By \cite[Theorem 18]{W2012}, there is a vertex $v$ of $\wh\Delta_i$ (and hence of $\wh\Delta$) whose vertex group admits a non-trivial free splitting relative to the incident cyclic edge groups. However, non-cyclic abelian groups do not split freely, and so this splitting is also relative to the incident non-cyclic abelian edges. But this contradicts the definition of $\wh{\Delta}$.
\end{proof}

The next result implies item (2).

\begin{proposition} \label{p:rigid}
If $\Delta_u$ is a rigid vertex group of $\Delta$ {(in the sense of Definition \ref{p:rigid})} then the corresponding vertex group $\widehat{\Delta}_u$ of $\widehat{\Delta}$ is {rigid}.
\end{proposition}
\begin{proof}
Consider an $\model_{\res}$--tree $T$ which has abelian edge stabilizers and is so that every noncyclic abelian subgroup is elliptic.  Consider the induced $L$--action coming from the canonical embedding of $L$ into $\model_{\res}$.  This is an action of the same type, and $\Delta_u$ is a rigid vertex group $\Delta$, so $\Delta_u$ fixes a point in $T$.  However, the group $\widehat{\Delta}_u$ was chosen so that it did not admit any essential abelian splittings relative to (the images of) the edge groups of $\Delta$ adjacent to $\Delta_u$.    Thus $\widehat{\Delta}_u$ fixes a point of $T$, as required.
\end{proof}

\begin{corollary}
The graph of groups $\widehat{\Delta}$ is the canonical primary JSJ decomposition of $\model_{\res}$.
\end{corollary}
\begin{proof}
{  We have proved that $\model_{\res}$ is freely indecomposable.  Therefore, Proposition \ref{p:JSJ props} implies that in order to show that $\widehat{\Delta}_{\res}$ is a canonical primary JSJ decomposition we have to verify Conditions (1)--(4) from that result.  Proposition \ref{p:rigid} shows that the vertices labelled `rigid' are indeed rigid.  Suppose that $\Delta_v$ is a (maximal) socket group of $\Delta$, with corresponding vertex group $\widehat{\Delta}_v$ of $\widehat{\Delta}$.  Clearly $\widehat{\Delta}_v$ is a socket, and we claim that it is a maximal socket. {  Suppose not.  Then there is an abelian vertex group adjacent to $\widehat{\Delta}_v$ of valence $2$ which is isomorphic to each of its edge groups, and the other nonabelian vertex group adjacent to this abelian vertex group is also a socket.  By construction, it cannot be a flexible socket, so it must be a rigid socket.  However, in this case, it arose from a vertex group of $\Delta$ which contained the boundary curve from $\widehat{\Delta}_v$, and also powers of the other two boundary curves.  This would means that the socket group $\Delta_v$ was not maximal, contrary to assumption.}

It is clear from the construction that $\widehat{\Delta}$ is bipartite with the required structure.  Moreover, by construction, any edge group adjacent to a nonabelian vertex group $\Delta_v$ is maximal abelian in $\Delta_v$.  It follows that $\widehat{\Delta}$ is $2$--acylindrical and so, by Proposition \ref{p:JSJ props}, $\widehat{\Delta}$ is the canonical primary JSJ decomposition of $\model_{\res}$, as required.
}
\end{proof}

We turn our attention to item (3).

\begin{proposition}\label{p: Mu is strict}
The map $\mu$ is strict.
\end{proposition}
\begin{proof}
There are three conditions to check, for the different flavours of vertex groups of $\widehat{\Delta}$.
 
For socket vertices, we must show that the image under $\mu$ is nonabelian, and this is immediate from the construction, because the socket vertices of $L$ embed in the socket groups of $\wh\Delta$ and $\lambda \co L \to L_1$ is strict.

Suppose now that $\widehat{\Delta}_{[w]}$ is an abelian vertex group of $\wh{\Delta}$, and let $\overline{P}(\wh{\Delta}_{[w]})$ be the peripheral subgroup.  Then $\overline{P}(\wh{\Delta}_{[w]}{)}$ is a colimit of the peripheral groups $\overline{P}(\Delta_{w_k})$ and the incident edge groups $\wh{\Delta}_{[e_i]}$ with maps from the edge groups of $\Delta$ to each of these.  That $\mu$ is injective on $\overline{P}(\wh{\Delta}_{[w]})$ follows from the injectivity of the maps from $\overline{P}(\Delta_{w_k})$ and $\wh{\Delta}_{[e_i]}$.

So it remains to check that $\mu$ is injective on the envelopes of rigid vertices.  We therefore consider a rigid vertex $v$ and its envelope $E=E(\wh{\Delta}_v)$.  Note that the Bass--Serre tree $\wh{T}$ of $\wh{\Delta}$ induces a splitting of $E$: the subgroup $E$ acts cocompactly on its minimal invariant subtree $S$, and the quotient graph is a star, with a central vertex labelled $\wh{\Delta}_v$ and  the leaves labelled by the peripheral subgroups of the adjacent abelian vertices.  

For convenience we write $M=\model_{\res_1}$ and consider the subgroup $N=\mu(E)$.  The hierarchy $\mc{T}_{M,\mu(\overline{\mc{C}}_v)}$ induces a hierarchy $\mc{T}_{N,\mu(\overline{\mc{C}}_v)}$ which, as in the paragraph after Theorem \ref{t: Strong accessibility},  defines a chain of subgroups
\[
N_n\subseteq N_{n-1}\subseteq\ldots\subseteq N_0=N
\]
and abelian $N_i$--trees $T_i$ so that $N_{i+1}$ is the stabilizer of a unique vertex $u_i$ of $T_i$.  By definition, the restriction of $\mu$ to $\wh{\Delta}_v$ is an isomorphism $\wh{\Delta}_v\to N_n$.   Let $E_i=\mu^{-1}(N_i)$ and let $S_i$ be the minimal $E_i$--invariant subtree of $S$.  

It is not hard to see that the $N_i$ are all freely indecomposable, and hence that the $N_i$--trees $T_i$ induced from Grushko trees are trivial. We therefore remove these redundancies from this list in the obvious way, and henceforth assume that each $T_i$ is induced from a primary JSJ tree.

We define canonical equivariant maps $q_i\co S_i\to T_i$ by sending the vertex of $S_i$ stabilized by $\wh{\Delta}_v$ to $u_i$, and by sending a vertex $w$ of $S_i$ with abelian stabilizer to the unique vertex of $T_i$ whose stabilizer is the maximal abelian subgroup of $N_i$ containing $\mu(\mathrm{Stab}_{E_i}(w))$.

We now prove that $\mu\co E\to N$ is injective by induction.  It suffices to show that if $\mu|_{E_i}\co E_i\to N_i$ is injective then $q_{i-1}\co S_{i-1}\to T_{i-1}$ does not factor through a fold.  Indeed, the malnormality of the family of edge groups incident at the central vertex of $S_{i-1}$ shows that a folded edge has non-abelian stabilizer, which contradicts the fact that edge stabilizers in $T_{i-1}$ are abelian.   This completes the proof.
\end{proof}

\begin{remark}\label{rem: Envelope hierarchy}
Consider a rigid vertex $v$ of $\wh{\Delta}$.  The proof of Proposition \ref{p: Mu is strict} uses the hierarchy induced on the image $N$ of the envelope $E=E(\wh{\Delta}_v)$ under the map $\mu$, which the proof shows is injective on $E$. The hierarchy $\mc{T}_{N,\mu(\overline{\mc{C}}_v)}$ for $N$ therefore yields a hierarchy $\mathcal{T}_{E,\overline{\mc{C}}_v}$ for $E$, which we will make use of later.  As in the proof of the proposition, we may assume that every level of $\mathcal{T}_{E,\overline{\mc{C}}_v}$ is induced by a JSJ tree.

\end{remark}

Item (4) is an immediate consequence of Theorem \ref{t:strict res implies limit group} and Proposition \ref{p: Mu is strict}.

\begin{corollary}\label{p:Models are limit groups}
The group $\model_\res$ is a $\Gamma$--limit group.
\end{corollary}

Before proving item (5), we first construct a map $\Phi\co\Mod(L)\to\Out(\model_\res)$.  Let $\alpha\in\Mod(L)$.  After suitably conjugating the natural maps $\rmodel_\res(\Delta_v)\to\model_\res$, we obtain a commutative diagram, and the universal property of $\model_\res$ guarantees a unique map $\wh{\alpha}\co\model_\res\to\model_\res$.  We set $\Phi(\alpha)=\wh{\alpha}$.

\begin{proposition}\label{prop: Hom of Mods}
The map $\Phi\co\Mod(L)\to\Out(\model_\res)$ is a homomorphism and, for $\alpha\in\Mod(L)$ and $g\in L$, we have that $\eta(\alpha(g))=\Phi(\alpha)(\eta(g))$.  Furthermore, its image lies in $\Mod(\model_\res)$.
\end{proposition}
\begin{proof}
 Uniqueness in the universal property of $\model_\res$ implies in the usual way that the assignment $\Phi(\alpha)=\wh{\alpha}$ is a homomorphism. From the commutativity of the diagram, 
 \[
\eta(\alpha(g))=\wh{\alpha}(\eta(g))
\]
for all $g\in L$, as required.

It therefore remains to prove that $\wh{\alpha}\in\Mod(\model_\res)$. By Definition \ref{d:Gen Dehn twist and modular group}, it suffices to show that if $\alpha\in\Mod(L)$ is a generalized Dehn twist then so is $\widehat{\alpha}\in\Out(\model_\res)$. We need to check the three cases of Definition \ref{d:Gen Dehn twist and modular group}. The second is immediate: if $\alpha$ is a Dehn twist in an essential two-sided simple closed curve in a socket vertex of $\Delta$, the same curve arises on a socket vertex of $\wh{\Delta}$, and we can also Dehn twist in that.

We deal with Cases (2) and (3) of Definition \ref{d:Gen Dehn twist and modular group} simultaneously, but first we need some notation.  For a group $G$ with a collection of subgroups $\mc{P}=\{P_1,\ldots,P_n\}$, we write $\Aut(G,\mc{P})$ for the subgroup of $\Aut(G)$ consisting of those automorphisms $\alpha$ so that, for each $i$, there is an inner automorphism $\iota_i\in\Aut(G)$ so that $\alpha|_{P_i}=\iota_i$.  Let $\Out(G,\mc{P})$ be the image of $\Aut(G,\mc{P})$ in $\Out(G)$.

Let $[w]$ be an abelian vertex of $\wh{\Delta}$ and let $[e]$ be an incident edge. Then we may assume that $\alpha$ is a generalized Dehn twist either coming from Case (2) of Definition \ref{d:Gen Dehn twist and modular group} and associated to an edge $e$ of $\Delta$, or coming from Case (3) of Definition \ref{d:Gen Dehn twist and modular group} and associated to an abelian vertex $w$ of $\Delta$.  Let $\Lambda_{[w]}$ be the graph of abelian groups that arises in the construction of $\wh{\Delta}_{[w]}$, as in \S\S\ref{s:Construction}.  

We next delete the equivalence class $[w]$, thought of as a set of abelian vertices of $\Delta$,  and let $\Delta_1,\ldots,\Delta_n$ be the graphs of groups corresponding to the resulting connected components. Let $P_1=\pi_1\Delta_i$ for each $i$, and let $\mc{P}=\{P_1,\ldots,P_n\}$, and note that $\alpha$ is contained in $\Out(L,\mc{P})$.

We next equip $\model_\res=\pi_1\wh{\Delta}$ with a similar family of subgroups.  Let $\wh{\Delta}_1,\ldots,\wh{\Delta}_m$ be the connected components obtained from $\wh{\Delta}$ after deleting the vertex $[w]$, let $Q_j=\pi_1\wh{\Delta}_j$ for each $j$, and let $\mc{Q}=\{Q_1,\ldots,Q_m\}$.  

Since $\wh{\Delta}_{[w]}$ is defined to be the abelianization of $\pi_1\Lambda_{[w]}$, there is a natural map $\Out(L,\mc{P})\to \Out(\model_\res,\mc{Q})$.   But the latter is generated by Dehn twists associated to edges incident at $[w]$ and by generalized Dehn twists associated to $[w]$.  In particular, $\alpha$ is indeed contained in $\Mod(\model_\res)$, as claimed.

\end{proof}

\section{Relatively immutable subgroups} \label{s:Rel Imm}

\subsection{Limiting actions on $\R$--trees and the Rips machine}

We recall from \cite{groves_rh,groves:limitRH2} that if $\Upsilon$ is a toral relatively hyperbolic group, $G$ is finitely generated group and $\{ \rho_i \co G \to \Upsilon \}$ a sequence of non-conjugate homomorphisms with stable kernel $K = \SK(\rho_i)$ then there is a limit action of $L = G/K$ on an $\R$--tree $T_\infty$ with no global fixed point, abelian edge stabilizers and trivial tripod stabilizers (see \cite[Theorem 6.5]{groves_rh}).  Since the $L$--action on $T_\infty$ is superstable, the Rips machine (see \cite[Main Theorem]{Guirardel:Rtrees}) applies and $L$ admits a nontrivial splitting over an abelian subgroup.

\begin{lemma} \label{l:splitting essential}
The group $L = G/\SK(\rho_i)$ described above admits a nontrivial primary splitting.
\end{lemma}
\begin{proof}
This follows from the way that the $L$--action on $T_\infty$ is turned into a simplicial action of $L$ on a tree.  There are essentially three cases to consider: axial, surface and simplicial.  The axial and surface cases are straightforward and the simplicial case follows quickly from the fact that tripod stabilizers are trivial.  One needs to be slightly careful using the tree of actions decomposition that \cite[Main Theorem]{Guirardel:Rtrees} produces, since the splitting may come from a point which is the intersection of two vertex trees, so the edge group needn't fix an arc in the tree.  This splitting may not be primary, but the classification of the vertex trees in this decomposition makes it straightforward to find a different nontrivial primary splitting of $L$.
\end{proof}

Below, when we consider relatively immutable subgroups, it will be useful to consider homomorphisms that act as inner automorphisms on particular subgroups.  To this end, we make the following definition.

\begin{definition}
Suppose that $G$ is a group, and $H \le G$ a subgroup.  Suppose further that $\mc{P}$ is a collection of subgroups of $H$.  Denote by $\Hom_{\mc{P}}(H,G)$ the set of maps $\lambda \co H \to G$ so that for each $P \in \mc{P}$ the restriction of $\lambda$ to $\mc{P}$ is the restriction of an inner automorphism of $G$ to $P$.
\end{definition}
Note that the inner automorphism associated to each $P$ might be different for different $P$.

\begin{lemma} \label{l:rel splitting}
Let $\Upsilon$ be a toral relatively hyperbolic group, $H$ a finitely generated subgroup and $\mc{P}$ a finite collection of finitely generated subgroups of $H$.
Suppose that $\{ \rho_i \}$ is a convergent sequence of non-conjugate homomorphisms from $\Hom_{\mc{P}}(H,\Upsilon)$ with trivial stable kernel.  Each element of $\mc{P}$ fixes a point in the limiting $\R$--tree $T_\infty$.

Moreover, $H$ admits a nontrivial primary splitting relative to $\mc{P}$.
\end{lemma}
\begin{proof}
Consider the homomorphisms $\{ \rho_i \co H \to \Upsilon \}$ as in the statement of the lemma, {let $P \in \mc{P}$ and suppose $p \in P$}.  Then $\rho_i(p)$ is always conjugate to $p$ in $\Upsilon$, which implies that there is some $x_i \in \Upsilon$ (the conjugating element) so that $d_\Upsilon(\rho_i(p).x_i,x_i) \le K$, where $K$ is the word length of $p$, and in particular doesn't depend on $i$.  By \cite[Lemma 5.4]{dahmanigroves1} this means that $p$ fixes a point in $T_\infty$.
If we take another $q \in P$ then it will fix the same point of $T_\infty$.  Applying this argument to a finite generating set of each $P \in \mc{P}$ shows that there are all elliptic in $T_\infty$.

It now follows from the relative version of the Rips machine in \cite{Guirardel:Rntrees} that $H$ admits a nontrival abelian splitting relative to $\mc{P}$, and a nontrivial primary relative  splitting can be found in the same way as for Lemma \ref{l:splitting essential} above.
\end{proof}

\subsection{Relatively immutable subgroups}

\begin{definition}
Let $G$ be a group, $H \le G$ a subgroup and $\mc{P}$ a collection of subgroups of $H$.  The pair $(H,\mc{P})$ is {\em relatively immutable in $G$} if the set $\Hom_{\mc{P}}(H,G)$ contains only finitely many conjugacy classes of injective homomorphisms.
\end{definition}

\begin{lemma} \label{l:rel immutable if no splitting}
Suppose that $\Upsilon$ is toral relatively hyperbolic, that $H \le \Upsilon$ is finitely generated and that $\mc{P}$ is a finite nonempty collection of nontrivial finitely generated subgroups of $H$.  Then $(H,\mc{P})$ is relatively immutable if and only if $H$ admits no nontrivial primary splitting relative to $\mc{P}$.
\end{lemma}
\begin{proof}
If $\Hom_{\mc{P}}(H,\Upsilon)$ contains infinitely many conjugacy classes of injective homomorphisms then Lemma \ref{l:rel splitting} implies that $H$ admits a nontrivial primary splitting relative to $\mc{P}$.

On the other hand, if $H$ does admit a nontrivial primary splitting relative to $\mc{P}$ then there are infinitely many conjugacy classes of automorphisms of $H$ which are contained in $\Hom_{\mc{P}}(H,H)$.  This can be seen, for example, by a straightforward adaptation to the relative case of \cite[Lemma 3.34]{dahmanigroves1}.
\end{proof}

The following result is similar to results in  \cite[$\S$ 4,5]{dahmanigroves1} and also to \cite[Lemma 7.5]{GrovesWilton10}.  In \cite{dahmanigroves1} it is assumed that the domain is finitely presented, and in \cite{GrovesWilton10} it is assumed that the target is hyperbolic and it is considering immutable rather than relatively immutable subgroups.  However, the technical details are similar.

\begin{proposition} \label{p:immutable on finite ball}
Suppose that $\Upsilon$ is a toral relatively hyperbolic group, that $H$ is a finitely generated subgroup of $\Upsilon$ and that $\mc{P}$ is a finite, nonempty collection of nontrivial finitely generated subgroups of $H$.  Then $(H,\mc{P})$ is relatively immutable if and only if there exist finitely many maps $\rho_1, \ldots , \rho_k
\in \Hom_{\mc{P}}(H,\Upsilon)$ and a positive integer $D$ so that for any map $\eta \in \Hom_{\mc{P}}(H,\Upsilon)$ either
\begin{enumerate}
\item $\eta$ is conjugate to $\rho_i$ for some $i$; or
\item $\eta$ is not injective on the ball of radius $D$ about $1$ in the Cayley graph of $H$ (with respect to the chosen finite generating set).
\end{enumerate}
\end{proposition}
\begin{proof}
If $(H,\mc{P})$ is not relatively immutable then there are infinitely many conjugacy classes of injective maps in $\Hom_{\mc{P}}(H,\Upsilon)$, so in this case there does not exist a $D$ and a finite list as in the statement of the result.

Conversely, suppose that the statement of the result does not hold.  Then there is a sequence
$\{ \phi_i \}$ of non-conjugate maps in $\Hom_{\mc{P}}(H,\Upsilon)$ so that $\phi_i$ is injective on the ball of radius $i$ about $1$ in $H$.  We can pass to a convergent subsequence of $\{ \phi_i \}$ and obtain a faithful action of $H$ on a limiting $\R$--tree $T_\infty$.  By Lemma \ref{l:rel splitting} $H$ admits a nontrivial primary splitting relative to $\mc{P}$ and so $(H,\mc{P})$ is not relatively immutable, by Lemma \ref{l:rel immutable if no splitting}.
\end{proof}

\subsection{Equations with rational constraints}

In this subsection we recall some results from \cite{dahmani09} and \cite{dahmanigroves1}, following the point of view of \cite{GrovesWilton10}.  What we do in this subsection is similar to \cite[$\S6$]{GrovesWilton10}, working with toral relatively hyperbolic groups instead of torsion-free hyperbolic groups.  In particular, see \cite[$\S2$]{GrovesWilton10} for an introduction to systems of equations and inequations over groups, and their relationship with homomorphisms.  In this section, we will always consider a solution to a system of equations as a homomorphism. 

Throughout this section, we fix a toral relatively hyperbolic group $\Upsilon$.  The main theorem in this section is Theorem \ref{t:fin many conjugacy classes} which asserts that we can algorithmically recognize when a finite system of equations and inequations over $\Upsilon$ has finitely many conjugacy classes of solutions (and list conjugacy--representatives of solutions in this case).  In fact, the algorithm is {\em uniform} in the sense that it takes a finite presentation for $\Upsilon$ as input, along with the equations and inequations.

Proving Theorem \ref{t:fin many conjugacy classes} involves understanding certain conjugacy classes of homomorphisms from a group $H_\Sigma$ (where $\Sigma = 1$ are the equations) to $\Upsilon$.  It is straightforward to check (using \cite[Theorem 0.1]{dahmani09}) whether or not every homomorphism from $H_\Sigma$ to $\Upsilon$ has abelian image, and in this case Theorem \ref{t:fin many conjugacy classes} is straightforward.  Therefore, we henceforth deal with situations where $H_\Sigma$ has nonabelian images in $\Upsilon$, and thereby can (effectively) choose a pair of elements $a,b \in H_\Sigma$ which do not commute (this will be witnessed by non-commuting images in $\Upsilon$).

\begin{definition}
Let $H$ be a group with a fixed generating set $X$.  A homomorphism $\psi \co H \to \Upsilon$ is {\em compatible} if it is injective on the ball of radius $8$ in $H$ (with respect to $X$).
\end{definition}

\begin{definition}
Let $k$ be a natural number and let $B_{F(X)}(k)$ be the ball of radius $k$ in the free group $F(X)$.  A system of equations and inequations $\Sigma = 1$, $\Lambda \ne 1$ {\em forces the ball of radius $k$} if there is a subset $S \subseteq B_{F(X)}(k)$ such that $S \subseteq \Sigma$ and $B_{F(X)}(k) \setminus S \subseteq \Lambda$.
\end{definition}

Suppose that $\Sigma =1$, $\Lambda \ne 1$ is a system of equations and inequations.  If $X$ is finite, there are finitely many subsets of $B_{F(X)}(8)$.  For any $S \subseteq B_{F(X)}(8)$, let $\Sigma_S = \Sigma \cup S$ and $\Lambda_S = \Lambda \cap \left( B_{F(X)}(8) \setminus S \right)$.  For any solution $\psi$ to $\Sigma =1$, $\Lambda \ne 1$ there is a unique $S$ so that $\psi$  is a solution to $\Sigma_S = 1$, $\Lambda_S \ne 1$.  Conversely, for any $S$, a solution to $\Sigma_S = 1$, $\Lambda_S \ne 1$ is obviously a solution to $\Sigma = 1$, $\Lambda \ne 1$.  
Moreover, for any $S$, any solution $\phi$ to $\Sigma_S = 1$, $\Lambda_S \ne 1$ corresponds to a compatible homomorphism $H_{\Sigma_S} \to \Upsilon$.

Therefore, by replacing $\Sigma = 1$, $\Lambda \ne 1$ by finitely many systems of equations and inequations, we may consider only solutions corresponding to compatible homomorphisms.

Fix once and for all a pair $a,b \in H$ of noncommuting elements, which exist if $H$ is nonelementary.  Given such $a$ and $b$, \cite[Remark 4.8]{dahmanigroves1} defines a condition $\Omega$, and \cite[Definition 4.5]{dahmanigroves1} (first defined in \cite{dahmani09}) defines the notion of an {\em acceptable lift}, which is a choice of words representing the images of the ball of radius $2$ under a homomorphism.  See \cite{dahmanigroves1}, \cite{dahmani09} for more details.

\begin{definition}
A homomorphism $\psi \co H_\Sigma \to \Upsilon$ is {\em fairly short} if it is compatible and has an acceptable lift $\tilde{\psi} \co B_{H_\Sigma}(2) \to F$ that satisfies $\Omega$.  A homomorphism $\psi \co H_\Sigma \to \Upsilon$ is {\em very short} if it is compatible and every acceptable lift $\tilde\psi \co B_{H_\Sigma}(2) \to F$ satisfies $\Omega$.
\end{definition}

The following result is implied by \cite[Lemma 4.6, Proposition 4.7]{dahmanigroves1}.

\begin{lemma}\label{lemma:Short homs}
Any conjugacy class of compatible homomorphism $H_\Sigma \to \Upsilon$ contains at least one very short homomorphism and at most finitely many fairly short homomorphisms.
\end{lemma}

As explained in \cite[$\S6$]{GrovesWilton10}, the following is now a consequence of \cite[Proposition 1.5]{dahmani09}.

These notions are useful because of the following proposition, which is an immediate consequence of a theorem of Dahmani \cite{dahmani09}. We refer the reader to \cite[Theorem 3.22]{dahmanigroves1} for the statement.

\begin{proposition}\label{prop:Short solutions with conjugation}
There exists an algorithm that takes as input:
\begin{enumerate}
\item a finite presentation for a toral relatively hyperbolic group $\Upsilon$ and a pair of non-commuting elements $a,b\in\Upsilon$,
\item two finite sets of variables $\x=\{x_i\},\y=\{y_j\}$,
\item a finite system (*) consisting of equations $\Sigma(\x) = 1$ and inequations $\Lambda(x) \ne 1$, and equations $\Theta(\x,\y)=1$ with coefficients in $\Upsilon$
\end{enumerate}
and always terminates with answer `Yes' or `No'.  In case the answer is `Yes', there exists a solution to (*) in $\Upsilon$ with an acceptable lift of $\x$ that satisfies $\Omega$.  In case the answer is `No', there is no solution to (*) in $\Upsilon$ for which every acceptable lift of $\x$ satisfies $\Omega$.
\end{proposition}

A special case of this result (where the variables $\y$ and the system $\Theta$ are not required) is the following.

\begin{proposition}\label{prop:Short solutions}
There exists an algorithm that, given a finite presentation for a toral relatively hyperbolic group, a finite system of equations $\Sigma = 1$ and a finite system of inequations $\Lambda \ne 1$, each with coefficients in $\Upsilon$, always terminates with answer `Yes' or `No'.  In case the answer is `Yes', there exists a fairly short solution to $\Sigma = 1$, $\Lambda \ne 1$ in $\Upsilon$.  In case the answer is `No', there does not exist any very short solution.
\end{proposition}

The following now has exactly the same proof as \cite[Theorem 6.9]{GrovesWilton10}.  We remark that the analogous result there is stated as if the algorithm depends on which group $\Upsilon$ is being considered.  However, as noted in \cite[Remark 1.5]{GrovesWilton10} all algorithms in that paper are uniform, in the sense that the finite presentation of $\Upsilon$ can be taken as input (with a single algorithm).  We state the uniform version below.

\begin{theorem} \label{t:fin many conjugacy classes}
There is an algorithm that takes as input a finite presentation for a toral relatively hyperbolic group $\Upsilon$, a finite system of equations and inequations $\Sigma = 1$, $\Lambda \ne 1$ and terminates if and only if there are finitely many conjugacy classes of solutions to the system.  In case the algorithm terminates, it outputs a list consisting of exactly one representative of each conjugacy class of solutions.
\end{theorem}
Note that the proof from \cite{GrovesWilton10} uses the solution to the simultaneous conjugacy problem in $\Upsilon$.  One way to see that the simultaneous conjugacy problem is solvable in toral relatively hyperbolic groups is that this is a special case of an equation, which can be solved by \cite[Theorem 0.1]{dahmani09}.

\begin{question}
In \cite{GrovesWilton10}, we proved that there exists a `complementary' algorithm to the one above, which terminates if and only if the system has infinitely many conjugacy classes of solutions.  Is there such a complementary result in the setting of toral relatively hyperbolic groups?
\end{question}

\subsection{Enumerating relatively immutable subgroups}

In this section we adapt the proof of \cite[Proposition 7.8]{GrovesWilton10} to enumerate \emph{relatively} immutable subgroups (in the toral relatively hyperbolic setting).  Immutable subgroups $H$ of a hyperbolic group $\Gamma$ were characterized as those subgroups for which there is a positive integer $D$ such that the set $\Hom(H,\Gamma)$ contained only finitely many $\Gamma$--conjugacy classes of homomorphisms injective on a ball of radius $D$.  We then recognized this condition by translating it into a set of equations and inequations over $\Gamma$, and using the rational constraint $\Omega$ to restrict conjugacy classes to finitely many representatives.

Our strategy here is similar. Proposition \ref{p:immutable on finite ball} characterizes relatively immutable subgroups $(H,\mc{P})$ of a toral relatively hyperbolic group $\Upsilon$ as those for which the set $\Hom_{\mc{P}}(H,\Upsilon)$ contains only finitely many $\Upsilon$--conjugacy classes injective on a ball of radius $D$ (for some $D$).  We will translate this into a system of equations and inequations over $\Upsilon$ and again use the rational constraint $\Omega$ to restrict conjugacy classes to finitely many representatives.

\begin{proposition} [cf. \cite{GrovesWilton10}, Proposition 7.8]\label{prop: Detect rel immutable}
There exists a Turing machine that takes as input a finite presentation $\langle X \mid R \rangle$ for a toral relatively hyperbolic group $\Upsilon$, a finite subset $\mc{A}$ of $\Upsilon$, and a tuple of finite subsets $\{ Q_1, \ldots , Q_k \}$ so that $Q_i \subseteq \langle \mc{A} \rangle$ for each $i$, and terminates if and only if $(\langle \mc{A} \rangle, \mc{Q})$ is relatively immutable in $\Upsilon$, where $\mc{Q} = \{ \langle Q_i \rangle \}_i$.
\end{proposition}
\begin{proof}
Let $H = \langle \mc{A} \rangle\leq \Upsilon$.  We run the following arguments in parallel for increasing values of a positive integer $D$.
For such a positive integer $D$, let $\wh{H}_D$ be the group generated by $\mc{A}$ with relations equal to all the loops in the ball of radius $D$ in $H$.  These relations can be enumerated using a solution to the word problem in $\Upsilon$.  Let $\{ Q^D_1, \ldots , Q^D_k \}$ be the tuples of elements of $\wh{H}_D$ corresponding to $\{ Q_1, \ldots , Q_k \}$.

A homomorphism $\wh{H}_D \to \Upsilon$ which is injective on the ball of radius $D$ in $\wh{H}_D$ can be characterized by a finite system of equations and inequations over $\Upsilon$.  The equations correspond to the relations in $\wh{H}_D$, and the inequations separate the elements of the ball of radius $D$ in $\wh{H}_D$.  To be precise, we take variables $\x$ corresponding to the elements of $\mc{A}$, a set of elements $\Sigma_D(\x)$ in the free group $F(\x)$ on $\x$ corresponding to the loops in the ball of radius $D$ in $\wh{H}_D$, and a set of elements $\Lambda_D(\x)\subseteq F(\x)$ corresponding to the differences between distinct elements of the ball of radius $D$ in $\wh{H}_D$.  Homomorphisms in $\Hom(\wh{H}_D,\Upsilon)$ which are injective on the ball of radius $D$ are then in natural bijection with solutions to the system
\begin{equation}\label{eqn:absolute}
\Sigma_D(\x)=1~,~\Lambda_D(\x)\neq 1
\end{equation}
over $\Upsilon$.

We now need to restrict attention to elements of $\Hom_{\mc{Q}_D}(\wh{H}_D,\Upsilon)$, where $\mc{Q}_D=\{\langle Q_1^D\rangle,\ldots,\langle Q_k^D\rangle\}$.    To do this, we add one variable $y_i$ for each $i$, together with extra equations which say that, for each $i$, the conjugate of each element of $Q^D_i$ by $y_i$ is equal to the corresponding element in $Q_i$.  To be precise, for each $i$, let $Q_i=\{q_{ij}\}\subseteq\Upsilon$ and for each $j$ let $\eta_{ij}(\x)$ be the corresponding element of $Q^D_i\subseteq \wh{H}_D$ (written as a word in the generators $\x$).  For each $i$ and $j$ we add the equation with coefficients
\begin{equation}
\Theta_D(\x,\y):=y_i\eta_{ij}(\x)y_i^{-1}q_{ij}^{-1}=1
\end{equation}
over $\Upsilon$ to the system (\ref{eqn:absolute}), to obtain a system that is denoted by (3).

We consider the projection map from the set of solutions to the system (3) to the set of solutions to the system (1) that we obtain by forgetting the values of the variables $y_i$.   When next argue that we can recognize when the image of this projection consists of finitely many $\Upsilon$--conjugacy classes.  

As long as $D\geq 8$ the system of equations (\ref{eqn:absolute}) forces the ball of radius 8, and therefore any solution necessarily corresponds to a compatible homomorphism.  We now apply the algorithm from Proposition \ref{prop:Short solutions with conjugation} to the system (3).  If it terminates with `No', there is no very short solution $\x$ to the system (\ref{eqn:absolute}) in the image of the projection map from the solutions to (3).   In particular, there are no conjugacy classes in the image of the projection, by Lemma \ref{lemma:Short homs}, and it follows that $\Hom_{\mc{Q}_D}(\wh{H}_D,\Upsilon)$ contains {no} conjugacy classes that are injective on the ball of radius $D$.

If the algorithm terminates with `Yes', then there is a fairly short solution to (1) in the image of the projection, and a simpleminded search will find a solution $(\x_0,\y_0)$ to (3) which projects to a fairly short solution to (1).  We now add further inequations (with coefficients) to (1), stipulating that $\x\neq \x_0$, and repeat the procedure.  

Since the natural quotient map $\wh{H}_D\to H$ induces an $\Upsilon$--equivariant injection $\Hom_{\mc{Q}}(H,\Upsilon)\into\Hom_{\mc{Q}_D}(\wh{H}_D,\Upsilon)$, if this algorithm terminates then there is a $D$ such that $\Hom_{\mc{Q}}(H,\Upsilon)$ contains only finitely many conjugacy classes that are injective on the ball of radius $D$, and hence $(H,\mc{Q})$ is relatively immutable. 

Conversely, suppose that  $(H,\mc{Q})$ is relatively immutable.  Then for large enough $D$ the conclusion of Proposition \ref{p:immutable on finite ball} is satisfied.
Since $\Upsilon$ is equationally Noetherian, for sufficiently large $D$ every homomorphism from $\wh{H}_D$ to $\Upsilon$ factors through the natural quotient map from $\wh{H}_D$ to $H$.  Therefore, for sufficiently large $D$ we know that the projection of the solutions of system (3) has only finitely many conjugacy classes.  It follows that the algorithm terminates exactly when $(H,\mc{Q})$ is relatively immutable, as required.
\end{proof}

In our enumeration of $\Gamma$--limit groups using models in Section \ref{s:Enumeration}, we will require the following elementary observations.

\begin{lemma} \label{l:still rel immutable}
Suppose that $\Upsilon$ is a toral relatively hyperbolic group, that $H \le \Upsilon$ and that $\mc{P}$ is a collection of nontrivial abelian subgroups of $H$.  If $(H,\mc{P})$ is relatively immutable then so is $(H,\mc{P'})$, where $\mc{P}' = \left\{ Z_H(P) \mid P \in \mc{P} \right\}$.

Also, if $(H,\mc{P})$ is relatively immutable, then so is $(H,\mc{P}_0)$, where $\mc{P}_0$ consists of a collection of $H$-conjugacy representatives of the elements of $\mc{P}$.
\end{lemma}

\section{Calculation of quasi-convex enclosures} \label{s:Calc QCE}

Suppose that $\Upsilon$ is a toral relatively hyperbolic group.  In this section, we explain how to take a finite subset $\mc{A}$ of $\Upsilon$, along with finitely many finite subsets $Q_1, \ldots, Q_k$ (all given as words in the generators of $\Upsilon$) and algorithmically calculate the quasi-convex enclosure of $\langle \mc{A} \rangle$ relative to $\{ \langle Q_1 \rangle , \ldots, \langle Q_k \rangle \}$.  Let $\mc{Q} = \{ \langle Q_1 \rangle , \ldots, \langle Q_k \rangle \}$.


The main difficulty is resolved by Theorem \ref{t:Find Splittings}, which asserts that we can compute \emph{relative} Grushko and JSJ decompositions for toral relatively hyperbolic groups. The absolute version of this theorem follows immediately from \cite[Theorem 1.4]{dahmanigroves2} and \cite[Theorem D]{dahmanigroves1}.    However, the strategy employed in \cite{dahmanigroves2} to compute Grushko decompositions (using connectivity of the Bowditch boundary) does not extend easily to the setting of relatively one-ended groups.  Instead, we use the following theorem, which enables us to handle free and abelian splittings simultaneously.

\begin{theorem}\cite[Theorem 18]{W2012}\label{thm:Shenitzer}
Let $\Upsilon$ be finitely generated, and the fundamental group of a graph of groups with infinite cyclic edge groups.  Then $\Upsilon$ is one-ended if and only if every vertex group is freely indecomposable relative to the incident edge groups.
\end{theorem}

Theorem \ref{thm:Shenitzer} generalizes a result of Shenitzer, and various similar statements and proofs have appeared in the literature \cite{shenitzer_decomposition_1955,diao_grushko_2005,FuPap06,louder_scott_????,touikan15}. The first author learned of Theorem \ref{thm:Shenitzer} from Sela, who intended to use it to prove Theorem \ref{t:Find Splittings} in the setting of torsion-free hyperbolic groups, as part of his unpublished proof of the isomorphism problem for those groups, and Fujiwara explained a proof to the second author.

\begin{theorem} \label{t:Find Splittings}
There is an algorithm which takes as input a finite presentation $\langle \mc{A} \mid \mc{R} \rangle$ of a toral relatively hyperbolic group $\Upsilon$ and a finite collection of finite tuples $\{Q_1,\ldots , Q_n\}$ of elements of $\Upsilon$ and outputs a Grushko decomposition for $\Upsilon$ relative to the subgroups $\mc{Q}_i=\langle Q_i\rangle$ and, for each free factor $\Upsilon_0$ which is freely indecomposable relative to those $\mc{Q}_i$ conjugate into $\Upsilon_0$, it also outputs the relative primary JSJ decomposition for $\Upsilon_0$.
\end{theorem}

The proof of Theorem \ref{t:Find Splittings} follows the strategy of the proof of \cite[Theorem D]{dahmanigroves1}, with two important modifications.  We outline the strategy and explain the modifications. 

The algorithm in parallel attempts to find non-trivial primary relative splittings of $\Upsilon$ (using the obvious adaptation of \cite[Theorem 5.15]{dahmanigroves1} to the relative setting) and to prove that $\Upsilon$ admits no such splittings using the algorithm from Proposition \ref{prop: Detect rel immutable}.  If a non-trivial splitting is found then we pass to the vertex groups, add the incident edge groups to the relative structure, and repeat.  By \cite{BF:access}, this procedure eventually terminates in a primary decomposition $\Lambda$ for $\Upsilon$ in which every vertex group is immutable relative to those $\mc{Q}_i$ that are conjugate into it and also to the incident edge groups.

Collapsing the non-trivial edge groups of $\Lambda$, one obtains a free splitting $\Lambda_{G}$ of $\Upsilon$ in which every vertex group is the fundamental group of a graph of groups with non-trivial abelian edge groups and relatively immutable vertex groups.  Note that $\Lambda_G$ is the relative Grushko decomposition of $\Upsilon$.  Indeed, if some vertex group $\Upsilon_v$ were to admit a relative free splitting then, since non-cyclic abelian subgroups would necessarily be elliptic in such a splitting, it would follow from Theorem \ref{thm:Shenitzer} that \emph{every} edge group incident at $\Upsilon_v$ would be elliptic.  Hence, $\Upsilon_v$ would admit a non-trivial relative free splitting, contradicting the relative immutability of $\Upsilon_v$.

Since we have now found the Grushko decomposition, we restrict attention to a freely indecomposable vertex group, which we shall for notational brevity also call $\Upsilon$.  The splitting $\Lambda$ of $\Upsilon$ that we have found has the property that every vertex group is relatively immutable; no vertex group admits a non-trivial relative primary splitting.  Such a decomposition is obtained from a relative primary JSJ of $\Upsilon$ by cutting each socket vertex along a maximal non-peripheral essential multicurve in the associated surface.  The procedure described after \cite[Proposition 6.3]{dahmanigroves1} explains how to reassemble the socket pieces into a primary JSJ decomposition as required.

Finally, we need to compute the \emph{canonical} primary JSJ decomposition $\Delta$ associated with $\Lambda$ by computing the associated tree of cylinders.  We explain how to do this in the following result, which completes the proof of Theorem  \ref{t:Find Splittings}.

\begin{proposition}[cf. \cite{dahmaniguirardel:iso}, Lemma 2.34]\label{prop:Compute tree of cylinders}
There is an algorithm that takes as input a presentation for a toral relatively hyperbolic group $\Upsilon$ and a primary graph-of-groups decomposition $\Lambda$ for $\Upsilon$ and outputs the decomposition $\Delta$ corresponding to the associated (collapsed) tree of cylinders.
\end{proposition}
\begin{proof}
Using the word problem in $\Upsilon$, we can determine which vertex groups of $\Lambda$ are abelian.  We list the non-abelian vertices $x_1,\ldots,x_k$ together with presentations of their associated stabilizers $\Delta_{x_i}=\Lambda_{x_i}$.  For each $i$, we also record the images of incident edge maps in $\Lambda_{x_i}$. 

We next compute conjugacy representatives of the cylinders of $\Lambda$ as follows.   For each pair of edges $e_1$ and $e_2$, we determine whether or not there exists $\gamma\in\Upsilon$ such that $\gamma\Lambda_{e_1}\gamma^{-1}$ commutes with $\Lambda_{e_2}$, by encoding this as a system of equations over $\Upsilon$.  If so, we write $e_1\sim e_2$.  We now introduce vertices $w_1,\ldots,w_l$, one for each equivalence class of edges $[e_j]$.  For each $w_j$, we choose a representative edge $e_j$ and label $w_j$ with the centralizer $\Delta_{w_j}=C(\Lambda_{e_j})$.  

Finally, we need to describe the edges of $\Delta$.  For each non-abelian vertex $x_i$, we partition the centralizers of incident edge groups into $\Delta_{x_i}$--conjugacy classes. Each such class is conjugate into a unique cylinder corresponding to some $w_j$.  These define the edge groups and the attaching maps.
\end{proof}

It is now straightforward to compute quasiconvex enclosures: one iteratively computes relative Grushko and JSJ decompositions, and Theorem \ref{t: Strong accessibility} ensures that the procedure terminates.  Therefore we have the following result.

\begin{proposition} \label{p:find enclosure}
There is an algorithm which takes as input a finite presentation $\langle \mc{A} \mid \mc{R} \rangle$ of a toral relatively hyperbolic group $\Upsilon$ and a finite subset $S$ of $\Upsilon$ (given as a set of words in $\mc{A}^{\pm}$) and outputs a finite presentation for the quasi-convex enclosure of $\langle S \rangle$ in $\Upsilon$.
\end{proposition}

\section{Enumerating $\Gamma$--limit groups} \label{s:Enumeration}

We now turn to the enumeration of $\Gamma$--limit groups. First, we need to describe the data that we will use to specify a $\Gamma$--limit group $L=L_0$. Our enumeration will be non-unique, and indeed our limit group $L_0$ will come equipped with a strict resolution $\res_0$.  The philosophy is to specify $L_0$ by specifying its model $M_0:=\model_{\res_0}$, together with certain nicely specified subsets of $M_0$ that generate $L_0$.  The data will also include the models of all the limit groups in the resolution $\res_0$, forming a strict resolution of $M_0$.

\begin{definition}
A \emph{model pair of (strict) resolutions} is a diagram of maps of strict resolutions of the form\\
\centerline{
\xymatrix{
   L_0  \ar@{>}[d]^{\eta_0} \ar@{>}[r]^{\lambda_0} & L_1 \ar@{>}[d]^{\eta_1} \ar@{>}[r]^{\lambda_1} & \cdots \ar@{>}[r]^{\lambda_{n-1}} & L_n \ar@{>}[d]^{\eta_n} \ar@{>}[r]^{\lambda_n} & \Gamma \ar@{>}[d]^{\mathrm{id}}  \\
 M_0 \ar@{>}[r]^{\mu_0}& M_1\ar@{>}[r]^{\mu_1} &  \cdots\ar@{>}[r]^{\mu_{n-1}}& M_n\ar@{>}[r]^{\mu_n}& \Gamma}}\\
where the first line is a strict resolution $\res_0$ of $L_0$ followed by a strict map $\lambda_n \co L_n \to \Gamma$, and the second line consists of the canonically associated resolution of models (along with a strict map $\mu_n \co M_n \to \Gamma$).  The {\em length} of this model pair is $n$.

In such a model pair of strict resolutions, the pair of $\Gamma$--limit groups $\eta_0\co L_0\to M_0$ is said to be a \emph{model pair of $\Gamma$--limit groups built over $\res_1$}, where $\res_1$ is the resolution of $L_1$ induced by truncating {$\res_0$}.
\end{definition}

By abuse of indices, we refer to the identity map $\Gamma \to \Gamma$ as a model pair of (strict) resolutions of length $-1$.  Note that by Proposition \ref{p:res of length 0}, when $\Gamma$ is considered with respect to the trivial resolution, the model of $\Gamma$ is $\Gamma$.

\begin{definition} \label{d:eff pair of res}
An \emph{effective pair of resolutions of length $n$} (specifying a resolution $\res_0$) consists of a commutative diagram of the form
\centerline{
\xymatrix{
   L_0  \ar@{>}[d]^{\eta_0} \ar@{>}[r]^{\lambda_0} & L_1 \ar@{>}[d]^{\eta_1} \ar@{>}[r]^{\lambda_1} & \cdots \ar@{>}[r]^{\lambda_{n-1}} & L_n \ar@{>}[d]^{\eta_n} \ar@{>}[r]^{\lambda_n} & \Gamma \ar@{>}[d]^{\mathrm{id}}  \\
 M_0 \ar@{>}[r]^{\mu_0}& M_1\ar@{>}[r]^{\mu_1} &  \cdots\ar@{>}[r]^{\mu_{n-1}}& M_n\ar@{>}[r]^{\mu_n}& \Gamma}}\\
where each $M_i$ is specified by a finite presentation $\langle\mc{A}_i\mid\mc{R}_i\rangle$, each $L_i$ is specified by a finite set $\mc{S}_i$, each map $\eta_i$ is specified by a map $\mc{S}_i\to  F(\mc{A}_i)$, each map $\mu_i\co M_i\to M_{i+1}$ is specified by a map $\mc{A}_i\to F(\mc{A}_{i+1})$, and each map $\lambda_i$ is specified by a map $\mc{S}_i\to F(\mc{S}_{i+1})$. 

An \emph{effective model pair of resolutions} is an effective pair of resolutions that defines a model pair.
\end{definition}

We remark that, on the face of it, there is no effective procedure for checking whether or not the maps $\lambda_i$ and $\mu_i$ that we specify are strict, or that $M_i$ is the model of $L_i$ (with respect to $\res_i$). However, we will describe a Turing machine that enumerates effective model pairs for which these properties do hold.

\begin{theorem}\label{thm: Enumerate Gamma-limit groups}
There is a Turing machine that takes as input a presentation for a torsion-free hyperbolic group $\Gamma$ and outputs a list of effective pairs of resolutions over $\Gamma$ such that:
\begin{itemize}
\item every effective resolution on the list defines a model pair of strict resolutions;
\item for every strict resolution $\res$ over $\Gamma$, every effective pair of resolutions that specifies the model pair canonically associated to $\res$ appears on the list.
\end{itemize}
\end{theorem}
\begin{proof}
We begin our enumeration with the unique effective model pair of resolutions of length $-1$, which is specified by the given presentation for $\Gamma$ (considered as the model), the given generating set for $\Gamma$ (considered as the $\Gamma$--limit group) and the identity map on this generating set (specifying the map from the $\Gamma$--limit group to its model).

We now provide a parallel enumeration of all effective model pairs of resolutions.  Suppose that we have an effective pair of resolutions of length $n-1$ (for some $n \ge 0$) given by\\
\centerline{
\xymatrix{
L_1 \ar@{>}[d]^{\eta_1} \ar@{>}[r]^{\lambda_1} & \cdots \ar@{>}[r]^{\lambda_{n-1}} & L_n \ar@{>}[d]^{\eta_n} \ar@{>}[r]^{\lambda_n} & \Gamma \ar@{>}[d]^{\mathrm{id}}  \\
M_1\ar@{>}[r]^{\mu_1} &  \cdots\ar@{>}[r]^{\mu_{n-1}}& M_n\ar@{>}[r]^{\mu_n}& \Gamma}}\\
along with data $\mc{A}_i$, $\mc{R}_i$ and maps as in Definition \ref{d:eff pair of res}. (We start the indices with $1$ so that we can append a model pair $\eta_0 \co L_0 \to M_0$ built over the model pair of resolutions.)

We now begin the enumeration of the model pairs $\eta_0 \co L_0 \to M_0$ built over the above effective pair.
We proceed to enumerate in parallel all of the possible collections of data as follows:

\begin{enumerate}
\item A finite bipartite graph $\Delta_0$ with {\em red} and {\em blue} vertices, and a choice of maximal tree in $\Delta_0$.
\item  A decomposition of the blue vertices of $\Delta_0$ into two disjoint families: the `socket' vertices, and the `rigid' vertices.  The red vertices of $\Delta_0$ are all `abelian' vertices.
\item For each edge $e$ of $\Delta_0$, a finite set of words $Q_e$ in $S_1^\pm$ which define nontrivial commuting elements of $M_1$.  If an edge $e$ is adjacent to a socket vertex of $\Delta_0$, the set $Q_e$ consists of a single element $q_e$.
\item For each rigid vertex $u$, with adjacent edges $e_1, \ldots , e_k$, a finite set of words $W_u$ in $S_1^\pm$ so that 
\begin{enumerate}
\item $Q_{e_1}, \ldots , Q_{e_k}$ are all in the subgroup $\langle W_u \rangle$; and
\item $\left( \langle W_u \rangle , \left\{ \langle Q_{e_i} \rangle \right\} \right)$ is relatively immutable in $M_1$.
\end{enumerate}
\item For a socket vertex $v$, a presentation of the socket of the fundamental group $\Sigma_v$ of a surface with boundary, with boundary components in bijection with certain powers of generators $q_e$ of edge groups for edges adjacent to $v$.  There is also a map from $\Sigma_v$ to $M_1$ with image lying in $\langle S_1 \rangle$ so that the generator of a boundary component maps to a power of the corresponding $q_e$ (under the bijection between boundary components and edges adjacent to $v$) and the image of $\Sigma_v$ is nonabelian.
\item For each edge $e$ of $\Delta_0$ we choose elements $\tau_e\in M_1$, and insist that $\tau_e=1$ if $e$ is in the maximal tree of $\Delta_0$.
\item For each abelian vertex group $w$, a finite set of words $W_w$ in $S_1^\pm$ which define nontrivial commuting elements of $M_1$ so that $\langle W_w \rangle$ contains each of the $\tau_eQ_e\tau_e^{-1 }$ for edges $e$ adjacent to $w$.
\end{enumerate}
Note that each such piece of data can be enumerated in a simple-minded manner, with the verification that it satisfies the appropriate property also being verified by the appropriate simple-minded enumeration, or an application of Proposition \ref{prop: Detect rel immutable} in the case of relatively immutable subgroups.

We will define a model $\Gamma$--limit group $M$ as the fundamental group of a graph of groups $\wh\Delta$ which will be constructed from the above data as in Section \ref{s:Gluing}.  The subgroup $L$ of which $M$ is the model will also be defined as the fundamental group of a graph of groups $\Delta$.  We will construct $\wh\Delta$ algorithmically.  However, $\Delta$ need not be equal to $\Delta_0$, and we do not know how to effectively construct $\Delta$.  On the other hand, we will obtain an explicit finite list of generators for $L$, as required.

We now define the graph of groups $\wh\Delta$ as follows.  Suppose that $u$ is a rigid vertex of $\Delta_0$, and consider the subgroup $\langle W_u \rangle$ of $M_1$.  Let $\wh\Delta_u$ be the quasiconvex enclosure of $\langle W_u \rangle$ in $M_1$.  A finite presentation for $\wh\Delta_u$ can be calculated by Proposition \ref{p:find enclosure}.  If $e$ is an edge in $\Delta$ adjacent to $u$, we define $\wh\Delta_e$ to be the centralizer in $\wh\Delta_u$ of $\langle Q_e \rangle$, which can be calculated by Theorem \ref{t:Compute Z(g)} (note that since $\wh\Delta_u$ is toral relatively hyperbolic, it is CSA, so the centralizer of $\langle Q_e \rangle$ is the centralizer of any nontrivial element of $Q_e$).

We define an equivalence relation on the edges of $\Delta_0$ by saying that two edges $e_1$ and $e_2$, adjacent to a rigid vertex $u$ are equivalent if $\wh\Delta_{e_1}$ and $\wh\Delta_{e_2}$ are conjugate in $\wh\Delta_u$.  Since these subgroups are nontrivial abelian subgroups, and $\wh\Delta_u$ is toral relatively hyperbolic, we can determine this equivalence using CSA and the algorithm to solve equations over $\wh\Delta_u$ (note that both $\wh\Delta_{e_1}$ and $\wh\Delta_{e_2}$ are maximal abelian subgroups of $\wh\Delta_u$, by construction).  

As in Section \ref{s:Gluing}, the graph of $\wh\Delta$ is the quotient graph of $\Delta_0$ obtained by identifying equivalent edges.  Each vertex in $\wh\Delta$ is naturally exactly one of rigid, socket or abelian.
The definition of the graph of groups $\wh\Delta$ is now exactly as in Section 4.  We have already defined the rigid vertex groups.  The socket vertex groups in $\wh{\Delta}$ are the sockets obtained from the socket vertices of $\Delta_0$ by attaching maximal roots in $M$ to generators of  boundary components.  The edge groups are centralizers in the same way as in Section \ref{s:Gluing}, and the abelian vertex groups are abelianizations of the fundamental group of a graph of groups exactly as in Section \ref{s:Gluing}.  Then we let $M = \pi_1 \wh\Delta$.

We next specify the finite set $S$ and the map $S\to M$ that defines $L$.  The data listed above naturally defines a graph of groups structure on $\Delta_0$. Note that we do \emph{not} have access to presentations for the vertex and edge groups of $\Delta_0$ (indeed, the vertex groups may not be finitely presentable). However, we may take $S$ to be a natural generating set for $\pi_1\Delta_0$ corresponding to the disjoint union of:
\begin{itemize}
\item the subsets $Q_e$ (as $e$ ranges over the edges of $\Delta_0$);
\item the $W_ u$ (as $u$ ranges over all rigid vertices);
\item the generating sets specified by the chosen presentation for $\Sigma_v$ (as $v$ ranges over all socket vertices);
\item the $W_w$  (as $w$ ranges over all abelian vertices);
\item the choices of stable letters $\tau_e$ for edges $e$ of $\Delta_0$ not in the chosen maximal tree.
\end{itemize}
The map $S\to M$ is now given by the construction of $M$. 

It follows immediately from Proposition \ref{p:Is a model} that $M$ is the model of $L$ (with respect to the resolution $L\to L_1\to \ldots L_n\to\Gamma$).

Enumerating these data enumerates all \emph{freely indecomposable} effective model pairs built over the above effective pair of resolutions.  To enumerate \emph{all} effective model pairs $L_0\to M_0$, we enumerate finite sets of such data, together with finitely generated free groups, and take their free products.  

Proceeding in this way, recursively and in parallel, over all effective pairs of resolutions that we build, we obtain a list of effective pairs of resolutions over $\Gamma$. It is clear that every model pair arises in this list, and we have shown that every pair of resolutions that we build is a model pair.
\end{proof}
 
We end by remarking that we do not know how to algorithmically decide if a given (strict) map $\lambda_n \co L_n \to \Gamma$ is an embedding of the freely indecomposable free factors of $L_n$, nor do we know how to algorithmically compute the JSJ decomposition of a given finitely generated subgroup of $\Gamma$.  Therefore, terminating with a strict map $\lambda_n \co L_n \to \Gamma$ is the best that we can do.  This introduces some extra redundancy in the list of resolutions, as in Remark \ref{rem:different resolution defn}, but we do not believe it affects the utility of the construction.

\end{document}